\documentclass[11pt]{amsart}
\usepackage{amsmath}
\usepackage{amscd}     
\usepackage{amsthm}        
\usepackage{amssymb}
\usepackage{latexsym} 
\usepackage{oldgerm}  
\usepackage{euscript}
\usepackage{epsfig}
\usepackage[width=16cm]{geometry} 
\usepackage{graphics}
\usepackage{array}
\usepackage{enumerate}  
\usepackage{epsf} 
\usepackage[all]{xy}  
\usepackage{color}  
\usepackage{hyperref}   
\usepackage{soul}

\newcommand{\comment}[1]{}

\newcommand{\be}{\begin{equation}}
\newcommand{\bel}[1]{\begin{equation}\label{#1}}
\newcommand{\qe}{\end{equation}}
\newcommand{\ee}{\end{equation}}
\newcommand{\eeq}{\end{equation}}
\newcommand{\ba}{\begin{eqnarray}}
\newcommand{\ea}{\end{eqnarray}}

\numberwithin{equation}{section}

\theoremstyle{plain}
\newtheorem{theorem}{Theorem}[section]

\theoremstyle{definition}
\newtheorem{definition}{Definition}[section]   
 
\theoremstyle{plain}   
\newtheorem{lemma}{Lemma}[section]
  
\theoremstyle{plain}
\newtheorem{proposition}{Proposition}[section]         
        
\theoremstyle{plain}
\newtheorem{corollary}{Corollary}[section]

\theoremstyle{definition}   
\newtheorem{remark}{Remark}[section]
\theoremstyle{definition} 
\newtheorem{example}{Example}[section]

\newcommand{\R}{{\mathbb R}}

\newcommand{\F}{{\mathbb F}}

\newcommand{\N}{{\mathbb N}}

\newcommand{\A}{{\mathcal A}}

\newcommand{\Mm}{{\mathcal M}}    

\newcommand{\Pp}{{\mathcal P}}

\newcommand{\Ss}{{\mathcal S}}

\newcommand{\Om}{{\Omega}}
\newcommand{\om}{{\omega}}

\newcommand{\vol}{{\rm vol}}

\newcommand{\pb}{{\mathbf p}}

\newcommand{\la}{\langle}
\newcommand{\ra}{\rangle}

\newcommand{\p}{{\partial}}

\newcommand{\sign}{\mbox{sign}}  
\renewcommand{\span}{\text{\rm span}}

\newcommand{\g}{{\mathfrak g}} 
\newcommand{\T}{{\mathbf T}}

\definecolor{brown}{RGB}{150,100,0}

\renewcommand{\sign}{{\rm sign}\,}

\begin{document}

\title{Congruent families and invariant tensors}
\author{Lorenz Schwachh\"ofer, 
Nihat Ay, J\"urgen Jost, H\^ong V\^an L\^e}
\address{L. Schwachh\"ofer, TU Dortmund University, Dortmund, Germany, lschwach@math.tu-dortmund.de 
\newline
N. Ay, J. Jost, Max-Planck-Institute for Mathematics in the Sciences, Leipzig, Germany, nay@mis.mpg.de, jjost@mis.mpg.de
\newline
H.V. L\^e, Academy of Sciences of the Czech Republic, Prague, hvle@math.cas.cz}

  
\begin{abstract}
Classical results of Chentsov and Campbell state that -- up to constant multiples -- the only $2$-tensor field of a statistical model which is  invariant under congruent Markov morphisms is  the Fisher metric  and the only invariant $3$-tensor field is the Amari-Chentsov tensor. We generalize this result for arbitrary degree $n$, showing that any family of $n$-tensors which is invariant under congruent Markov morphisms is algebraically generated by the canonical tensor fields defined in \cite{AJLS2015}.
\end{abstract}
            
\maketitle

\section{Introduction} \label{sec:intro}

The main task of {\em Information geometry }is to use differential geometric methods in probability theory in order to gain insight into the structure of families of probability measures or, slightly more general, finite measures on some (finite or infinite) sample space $\Om$. In fact, one of the key themes of differential geometry is to identify quantities that do not depend on how we parametrize our objects, but that depend only on their intrinsic structure. And since in information geometry, we not only have the structure of the parameter space, the classical object of differential geometry, but also the sample space on which the probability measures live, we should also look at invariance properties with respect to the latter. That is what we shall systematically do in this contribution.

When parametrizing such a family by a manifold $M$, there are two classically known symmetric tensor fields on the parameter space $M$. The first is a quadratic form (i.e., a Riemannian metric), called the {\em Fisher metric $\g^F$}, and the second is a $3$-tensor, called the {\em Amari-Chentsov tensor $\T^{AC}$}. The Fisher metric was first suggested by Rao \cite{Rao1945}, followed by Jeffreys \cite{Jeffreys1946}, Efron \cite{Efron1975} and then systematically developed by Chentsov and Morozova \cite{Chentsov1965}, \cite{Chentsov1978} and \cite{MC1991}; the Amari-Chentsov tensor and its significance was discovered by Amari \cite{Amari1980}, \cite{Amari1982} and Chentsov \cite{Chentsov1982}. If the family is given by a positive density function $\pb(\xi) = p(\cdot; \xi) \mu$ w.r.t. some fixed background measure $\mu$ on $\Om$ and $p: \Om \times M \to (0, \infty)$ differentiable in the $\xi$-direction, then the score
  \begin{equation}\label{score}
\displaystyle{\int_\Om \p_V \log p(\cdot;\xi)\;  d\pb(\xi)} =0
\end{equation}
vanishes, while the Fisher metric $\g^F$ and the Amari-Chentsov tensor $\T^{AC}$ associated to a parametrized measure model are given by
\begin{equation} \label{Fisher-AC}
\begin{array}{lll}
\g^F(V, W) & := & \displaystyle{\int_\Om \p_V \log p(\cdot;\xi)\; \p_W \log p(\cdot;\xi)\; d\pb(\xi)}\\
\T^{AC}(V, W, U) & := & \displaystyle{\int_\Om \p_V \log p(\cdot;\xi)\; \p_W \log p(\cdot;\xi)\; \p_U \log p(\cdot;\xi)\; d\pb(\xi)}.
\end{array}
\end{equation}
Of course, this naturally suggests to consider analogous tensors for arbitrary degree $n$.
The tensor fields in \eqref{Fisher-AC} have some remarkable properties. On the one hand, they may be defined independently of the particular choice of a parametrization and thus are naturally defined from the differential geometric point of view. Their most important property from the point of view of statistics is that these tensors are invariant under sufficient statistics or, more general by congruent Markov morphisms. In fact, these tensor fields are characterized by this invariance property. This was shown in the case of finite sample spaces by Chentsov in \cite{Chentsov1978} and for an arbitrary sample space by the authors of the present article in \cite{AJLS}.

The question addressed in this article is to classify all tensor fields which are invariant under sufficient statistics and congruent Markov morphisms. In order to do this, we first have to make this invariance condition precise.

Observe that both \cite{Chentsov1978} and \cite{AJLS} require the family to be of the form $\pb(\xi) = p(\cdot; \xi) \mu$ with $p > 0$, which in particular implies that all these measures are equivalent, i.e., have the same null sets. Later, in \cite{AJLS2015} and \cite{AJLS2017}, the authors of this article introduced a more general notion of a {\em parametrized measure model }as a map $\pb: M \to \Mm(\Om)$ from a (finite or infinite dimensional) manifold $M$ into the space $\Mm(\Om)$ of finite measures which is continuously Fr\'echet-differentiable when regarded as a map into the Banach lattice $\Ss(\Om) \supset \Mm(\Om)$ of {\em signed } finite measures. Such a model neither requires the existence of a measure dominating all measures $\pb(\xi)$, nor does it require all these measures to be equivalent.

Furthermore, for each $r \in (0,1]$ there is a well defined Banach lattice $\Ss^r(\Om)$ of {\em $r$-th powers of finite signed measures}, whose nonnegative elements are denoted by $\Mm^r(\Om) \subset \Ss^r(\Om)$, and for each integer $n \in \N$, there is a {\em canonical $n$-tensor }on $\Ss^{1/n}(\Om)$ given by
\begin{equation} \label{Ln-def}
L_n^\Om(\nu_1, \ldots, \nu_n) := n^n \ (\nu_1 \cdots \nu_n)(\Om),
\end{equation}
where $\nu_1 \cdots \nu_n \in \Ss(\Om)$ is a signed measure. The multiplication on the right hand side of (\ref{Ln-def}) refers to the multiplication of roots of measures, cf. \cite[(2.11)]{AJLS2015}, see also (\ref{eq:prod-meas}). A parametrized measure model $\pb: M \to \Mm(\Om)$ is called {\em $k$-integrable }for $k \geq 1$ if the map
\[
\pb^{1/k}: M \longrightarrow \Mm^{1/k}(\Om) \subset \Ss^{1/k}(\Om), \qquad \xi \longmapsto \pb(\xi)^{1/k}
\]
is continuously Fr\'echet differentiable, cf. \cite[Definition 4.4]{AJLS2015}. In this case, we define the {\em canonical $n$-tensor of the model }as the pull-back $\tau^n_{(M, \Om, \pb)} := (\pb^{1/n})^\ast L_n^\Om$ for all $n \leq k$. If the model is of the form $\pb(\xi) = p(\cdot; \xi) \mu$ with a positive density function $p > 0$, then
\begin{equation} \label{tau-n}
\tau^n_{(M, \Om, \pb)} (V_1, \ldots, V_n) := \int_\Om \p_{V_1} \log p(\cdot;\xi)\; \cdots \p_{V_n} \log p(\cdot;\xi)\; d\pb(\xi),
\end{equation}
so that $\g^F = \tau^2_{(M, \Om, \pb)}$ and $\T^{AC} = \tau^3_{(M, \Om, \pb)}$ by (\ref{Fisher-AC}). The condition of $k$-integrability ensures that the integral in (\ref{tau-n}) exists for $n \leq k$.

A Markov kernel $K: \Om \to \Pp(\Om')$ induces a bounded linear map $K_\ast: \Ss(\Om) \to \Ss(\Om')$, called the {\em Markov morphism associated to $K$}. This Markov kernel is called {\em congruent}, if there is a statistic $\kappa: \Om' \to \Om$ such that $\kappa_\ast K_\ast \mu = \mu$ for all $\mu \in \Ss(\Om)$. 

We may associate to $K$ the map $K_r: \Ss^r(\Om) \to \Ss^r(\Om')$ by $K_r(\mu_r) = (K_\ast(\mu_r^{1/r}))^r$, where $\mu_r^{1/r} \in \Ss(\Om)$. While $K_r$ is not Fr\'echet differentiable in general, we still can define in a natural way the {\em formal differential }$dK_r$ and hence the pullback $K_r^\ast \Theta^n_{\Om;r}$ for any covariant $n$-tensor on $\Ss^r(\Om')$ which yields a covariant $n$-tensor on $\Ss^r(\Om)$.

It is not hard to show that for the canonical tensor fields we have the identity $K_{1/n}^\ast L_n^{\Om'} = L_n^\Om$ for any congruent Markov kernel $K: \Om \to \Pp(\Om')$, whence we may say that the canonical $n$-tensors $L_n^\Om$ on $\Ss^{1/n}(\Om)$ form a {\em congruent family}. Evidently, any tensor field which is given by linear combinations of tensor products of canonical tensors and permutations of the argument is also a congruent family, and the families of this type are said to be {\em algebraically generated by $L^n_\Om$}.

Our main result is that these exhaust the possible invariant families of covariant tensor fields:

\begin{theorem} \label{Maintheorem} Let $(\Theta^n_{\Om;r})$ be a family of covariant $n$-tensors on $\Ss^r(\Om)$ for each measurable space $\Om$. Then this family is invariant under congruent Markov morphisms if and only if it is algebraically generated by the canonical tensors $L_m^\Om$ with $m \leq 1/r$.

In particular, on each $k$-integrable parametrized measure model $(M, \Om, \pb)$ any tensor field which is invariant under congruent Markov morphisms is algebraically generated by the canonical tensor fields $\tau_{(M, \Om, \pb)}^m$, $m \leq k$.
\end{theorem}

We shall show that this conclusion already holds if the family is invariant under congruent Markov morphisms $K: I \to \Pp(\Om)$ with {\em finite }$I$. Also, observe that this theorem yields another proof of the theorems of Chentsov \cite[Theorem 11.1]{Chentsov1982} and Campbell (\cite{Campbell1986} or \cite{AJLS}) which classify the invariant families of $2$- and $3$-tensors, respectively. Campbell's theorem covers the case where the measures no longer need to be probability measure. In such a situation, the analogue of the score \eqref{score} no longer needs to vanish, and it furnishes a nontrivial $1$-tensor.

Let us comment on the relation of our results to those of Bauer et al. \cite{BBM} \cite{BBM2}. Assuming that the sample space $\Om$ is a manifold (with boundary or even with corners), the space $\text{Dens}_+(\Om)$ of {\em (smooth) densities on $\Om$} is defined as the set of all measures of the form $\mu = f \vol_g$, where $f > 0$ is a smooth function with finite integral, $\vol_g$ being the volume form of some Riemannian metric $g$ on $M$. Thus, $\text{Dens}_+(\Om)$ is a Fr\'echet manifold, and regarding a diffeomorphism $K: \Om \to \Om$ as a congruent statistic, the induced maps $K_r: \text{Dens}_+(\Om)^r \to \text{Dens}_+(\Om)^r$ are diffeomorphisms of Fr\'echet manifolds. The main result in \cite{BBM} states that for $\dim \Om \geq 2$ any $2$-tensor field which is invariant under diffeomorphisms is a multiple of the Fisher metric. Likewise,  the space of  diffeomorphism invariant $n$-tensors for arbitrary $n$ \cite{BBM2} is generated by the canonical tensors. Thus, when restricting to parametrized measure models $\pb: M \to \text{Dens}_+(\Om) \subset \Mm(\Om)$ whose image lies in the space of densities and which are differentiable w.r.t. the Fr\'echet manifold structure on $\text{Dens}_+(\Om)$, then the invariance of a tensor field under diffeomorphisms rather than under arbitrary congruent Markov morphisms already implies that the tensor field is algebraically generated by the canonical tensors. Considering invariance under diffeomorphisms is natural in the sense that they can be regarded as the natural analogues of permutations of a finite sample space. In our more general setting, however, the concept of a diffeomorphism is no longer meaningful, and we need to consider invariance under a larger class of transformations, the congruent Markov morphisms.

In a similar spirit, J. Dowty \cite{Dowty} has shown recently that when restricting to the space of {\em exponential families}, the Fisher metric is the only $2$-tensor which is invariant under independent and identically distributed extensions and canonical sufficient statistics.

This paper is structured as follows. In Section \ref{sec:prelims} we recall from \cite{AJLS2015} the definition of a parametrized measure model, roots of measures and congruent Markov kernels, and furthermore we give an explicit description of the space of covariant families which are algebraically generated by the canonical tensors. In Section \ref{sec:congruent-general} we recall the notion of congruent families of tensor fields and show that the canonical tensors and hence tensors which are algebraically generated by these are congruent. Then we show that these exhaust all invariant families of tensor field on {\em finite }sample spaces $\Om$ in Section \ref{sec:congruent-finite}, and finally, in Section \ref{sec:congruent-arb}, by reducing the general case to the finite case through step function approximations, we obtain the classification result Theorem \ref{thm:chentsov-classif} which implies Theorem \ref{Maintheorem} as a simplified version.

\noindent {\bf Acknowledgements.} This work was mainly carried out at the Max Planck Institute for Mathematics in the Sciences in Leipzig, and we are grateful for the excellent working conditions provided at that institution. H.V. L\^e  is partially supported by Grant RVO:67985840.  

\section{Preliminary results} \label{sec:prelims}

\subsection{The space of (signed) finite measures and their powers}

Let  $(\Om, \Sigma)$ be a measurable space, that is an arbitrary set $\Om$ together with a sigma algebra $\Sigma$ of subsets of $\Om$. Regarding the sigma algebra $\Sigma$ on $\Om$ as fixed, we let
\begin{equation} \label{eq:def-SsOm}
\begin{array}{lll}
\Pp(\Om) & := & \{ \mu \;:\; \mu\; \mbox{a probability measure on $\Om$}\}\\[2mm]
\Mm(\Om) & := & \{ \mu \;:\; \mu\; \mbox{a finite measure on $\Om$}\}\\[2mm]
\Ss(\Om) & := & \{ \mu \;:\; \mu\; \mbox{a signed finite measure on $\Om$}\}\\[2mm]
\Ss_a(\Om) & := & \{ \mu \in \Ss(\Om) \;:\; \int_\Om d\mu = a\}.
\end{array}
\end{equation}

Clearly, $\Pp(\Om) \subset \Mm(\Om) \subset \Ss(\Om)$, and $\Ss_0(\Om), \Ss(\Om)$ are real vector spaces, whereas $\Ss_a(\Om)$ is an affine space with linear part $\Ss_0(\Om)$. In fact, both $\Ss_0(\Om)$ and $\Ss(\Om)$ are Banach spaces whose norm is given by the total variation of a signed measure, defined as
\[
    \|\mu\| \; := \; \sup \sum_{i = 1}^n |\mu(A_i)|
\] 
where the supremum is taken over all finite partitions $\Omega = A_1 \dot\cup \dots \dot\cup A_n$ with disjoint sets $A_i \in \Sigma$. Here, the symbol $\dot \cup$ stands for the disjoint union of sets. In particular,
\[
    \|\mu\| \; = \; \mu(\Om) \qquad \mbox{for $\mu \in \Mm(\Om)$}.
\] 

In \cite{AJLS2015}, for each $r \in (0,1]$ the space $\Ss^r(\Om)$ of {\em $r$-th powers of measures on $\Om$} is defined. We shall not repeat the formal definition here, but we recall the most important features of these spaces.

Each $\Ss^r(\Om)$ is a Banach lattice whose norm we denote by $\|\cdot\|_{\Ss^r(\Om)}$, and $\Mm^r(\Om) \subset \Ss^r(\Om)$ denotes the spaces of nonnegative elements. Moreover, $\Ss^1(\Om) = \Ss(\Om)$ in a canonical way. For $r,s,r+s \in (0,1]$ there is a bilinear product
\begin{equation} \label{eq:prod-meas}
\cdot: \Ss^r(\Om) \times \Ss^s(\Om) \longrightarrow \Ss^{r+s}(\Om) \quad \mbox{such that} \quad \|\mu_r \cdot \mu_s\|_{\Ss^{r+s}(\Om)} \leq \|\mu_r\|_{\Ss^r(\Om)} \|\mu_s\|_{\Ss^s(\Om)},
\end{equation}
and for $0 < k < 1/r$ there is a exponentiating map $\pi^k: \Ss^r(\Om) \to \Ss^{kr}\Om)$ which is continuous for $k < 1$ and a Fr\'echet-$C^1$-map for $k \geq 1$.

In order to understand these objects more concretely, let $\mu \in \Mm(\Om)$ be a measure, so that $\mu^r := \pi^r(\mu) \in \Ss^r(\Om)$. Then for all $\phi \in L^{1/r}(\Om, \mu)$ we have $\phi \mu^r \in \Ss^r(\Om)$, and $\phi \mu^r \in \Mm^r(\Om)$ if and only if $\phi \geq 0$. The inclusion
\[
\Ss^r(\Om; \mu) := \{ \phi \mu^r \mid \phi \in L^{1/r}(\Om, \mu)\} \hookrightarrow \Ss^r(\Om)
\]
is an isometric inclusion of Banach spaces, and the elements of $\Ss^r(\Om, \mu)$ are said to be {\em dominated by $\mu$}. We also define
\[
\Ss^r_0(\Om; \mu) := \{ \phi \mu^r \mid \phi \in L^{1/r}(\Om, \mu), {\mathbb E}_\mu(\phi) = 0\} \subset \Ss^r(\Om; \mu).
\]
Moreover,
\begin{equation} \label{eq:product-concrete}
(\phi \mu^r) \cdot (\psi \mu^s) = (\phi \psi) \mu^{r+s}, \qquad \pi^k(\phi \mu^r) := \sign(\phi) |\phi|^k \mu^{rk},
\end{equation}
where $\phi \in L^{1/r}(\Om, \mu)$ and $\psi \in L^{1/s}(\Om, \mu)$. The Fr\'echet derivative of $\pi^k$ at $\mu_r \in \Ss^r(\Om)$ is given by
\begin{equation} \label{eq:deriv-pik}
d_{\mu_r} \pi^k(\nu_r) = k\; |\mu|^{k-1} \cdot \nu_r.
\end{equation}

Furthermore, for an integer $n \in \N$, we have the {\em canonical $n$-tensor on $\Ss^{1/n}(\Om)$}, given by
\begin{equation} \label{eq:canonical-nform}
L^n_\Om(\mu_1, \ldots, \mu_n) := n^n \int_\Om d(\mu_1 \cdots \mu_n) \qquad \mbox{for $\mu_i \in \Ss^{1/n}(\Om)$},
\end{equation}
which is a symmetric $n$-multilinear form, where we regard the product $\mu_1 \cdots \mu_n$ as an element of $\Ss^1(\Om) = \Ss(\Om)$. For instance, for $n = 2$ the bilinear form
\[
\la \cdot; \cdot\ra := \dfrac14 L^2_\Om(\cdot, \cdot)
\]
equips $\Ss^{1/2}(\Om)$ with a Hilbert space structure with induced norm $\|\cdot\|_{\Ss^{1/2}(\Om)}$.

\subsection{Parametrized measure models} \label{sec:pmm}

Recall from \cite{AJLS2015} that a parametrized measure model is a triple $(M, \Om, \pb)$ consisting of a (finite or infinite dimensional) manifold $M$ and a map $\pb: M \to \Mm(\Om)$ which is Fr\'echet-differentiable when regarded as a map into $\Ss(\Om)$ (cf. \cite[Definition 4.1]{AJLS2015}). If $\pb(\xi) \in \Pp(\Om)$ for all $\xi \in M$, then $(M, \Om, \pb)$ is called a {\em statistical model}. Moreover, $(M, \Om, \pb)$ is called {\em $k$-integrable}, if $\pb^{1/k}: M \to \Ss^{1/k}(\Om)$ is also Fr\'echet integrable (cf. \cite[Definition 2.6]{CRa}). For a parametrized measure model, the differential $d_\xi \pb(v) \in \Ss(\Om)$ with $v \in T_\xi M$ is always dominated by $\pb(\xi) \in \Mm(\Om)$, and we define the {\em logarithmic derivative} (cf. \cite[Definition 4.3]{AJLS2015}) as the Radon-Nikodym derivative
\begin{equation} \label{eq:log-der}
\p_v \log \pb(\xi) := \dfrac{d\{d_\xi \pb(v)\}}{d\pb(\xi)} \in L^1(\Om, \pb(\xi)).
\end{equation}
Then $\pb$ is $k$-integrable if and only if $\p_v \log \pb \in L^k(\Om, \pb(\xi)$ for all $v \in T_\xi M$, and the function $v \mapsto \|\p_v \log \pb\|_{\p_v \log \pb(\xi)}$ on $TM$ is continuous (cf. \cite[Theorem 2.7]{CRa}). In this case, the Fr\'echet derivative of $\pb^{1/k}$ is given as
\begin{equation} \label{eq:formal-derivative}
d_\xi \pb^{1/k}(v) = \dfrac1k \p_v \log \pb(\xi) \pb^{1/k}.
\end{equation}

\subsection{Congruent Markov morphisms}

\begin{definition} \label{def:markov}
A {\em Markov kernel } between two measurable spaces $(\Omega, \mathfrak B)$ and $(\Omega', \mathfrak B')$ is a map $K: \Om \to \Pp(\Om')$ associating to each $\om \in \Om$ a probability measure on $\Om'$ such that for each fixed measurable $A' \subset \Om'$ the map
\[
\Om \longrightarrow [0,1], \qquad \om \longmapsto K(\om) (A') =: K(\om;A')
\]
is measurable for all $A' \in \mathfrak B'$. The linear map
\begin{equation} \label{eq:Markov-linear}
K_\ast: \Ss(\Om) \longrightarrow \Ss(\Om'), \qquad K_*\mu(A') := \int_\Om K(\om; A')\; d\mu(\om)
\end{equation}
is called the {\em Markov morphism induced by $K$}.
\end{definition}

Evidently, a Markov morphism maps $\Mm(\Om)$ to $\Mm(\Om')$, and
\begin{equation} \label{eq:Markov-preserve}
\|K_\ast \mu\| = \|\mu\| \qquad \mbox{ for all $\mu \in \Mm(\Om)$,}
\end{equation}
so that $K_\ast$ also maps $\Pp(\Om)$ to $\Pp(\Om')$. For any $\mu \in \Ss(\Om)$, $\|K_\ast \mu\| \leq \|\mu\|$, whence $K_\ast$ is bounded.

\begin{example} \label{ex:MarKerFin}
A measurable map $\kappa: \Om \to \Om'$, called a {\em statistic}, induces a Markov kernel by setting $K^\kappa(\om) := \delta_{\kappa \om} \in \Pp(\Om')$. In this case,
\[
K^\kappa_\ast \mu(A') = \int_\Om K^\kappa(\om; A')\; d\mu(\om) = \int_{\kappa^{-1}(A')} d\mu = \mu(\kappa^{-1} A') = \kappa_\ast \mu(A'),
\]
whence $K^\kappa_\ast \mu = \kappa_\ast \mu$ is the push-forward of (signed) measures on $\Om$ to (signed) measures on $\Om'$.
\end{example}

\begin{definition} \label{def:CongMarkov}
A Markov kernel $K: \Om \to \Pp(\Om')$ is called {\em congruent w.r.t. to the statistic $\kappa: \Om' \to \Om$} if
\[
\kappa_*K(\om) = \delta_{\om} \qquad \mbox{for all $\om \in \Om$},
\]
or, equivalently, if $K_*$ is a right inverse of $\kappa_*$, i.e., $\kappa_* K_* = {\rm Id}_{\Ss(\Om)}$. It is called  {\em congruent }if it is congruent w.r.t. some statistic $\kappa: \Om' \to \Om$.
\end{definition}

This notion was introduced by Chentsov in the case of finite sample spaces \cite{Chentsov1982}, but the natural generalization in Definition \ref{def:CongMarkov} to arbitrary sample spaces has been treated in \cite{AJLS}, \cite{AJLS2015}  and \cite{Le2016}.

\begin{example} \label{ex:congstat}
A statistic $\kappa: \Om \to I$ between finite sets induces a partition
\[
\Om = \dot \bigcup_{i \in I} \Om_i, \qquad \mbox{where} \qquad \Om_i = \kappa^{-1}(i).
\]
In this case, a Markov kernel $K: I \to \Pp(\Om)$ is $\kappa$-congruent of and only of
\[
K(i)(\Om_j) = K(i; \Om_j) = 0 \qquad \mbox{for all $i \neq j \in I$.}
\]
\end{example}

If $(M, \Om, \pb)$ is a parametrized measure model and $K: \Om \to \Pp(\Om')$ a Markov kernel, then $(M, \pb', \Om')$ with $\pb' := K_\ast \pb: M \to \Mm(\Om') \subset \Ss(\Om')$ is again a parametrized measure model. In this case, we have the following result.

\begin{proposition} \label{prop:formal-p1/k} (\cite[Theorem 3.3]{AJLS2015})
Let $K_\ast: \Ss(\Om) \to \Ss(\Om')$ be a Markov morphism induced by the Markov kernel $K: \Om \to \Pp(\Om')$, let $\pb: M \to \Mm(\Om)$ be a $k$-integrable parametrized measure model and $\pb' := K_\ast \pb: M \to \Mm(\Om')$. Then $\pb'$ is also $k$-integrable, and
\begin{equation} \label{eq:dp'-compare}
\|\p_v \log \pb'(\xi)\|_{L^k(\Om', \pb'(\xi))} \leq \|\p_v \log \pb(\xi)\|_{L^k(\Om, \pb(\xi))}.
\end{equation}
\end{proposition}

\subsection{Tensor algebras} \label{sec:algebra}

In this section we shall provide the algebraic background on tensor algebras. Let $V$ be a vector space over a commutative field $\F$, and let $V^\ast$ be its dual. The {\em tensor algebra of $V^\ast$} is defined as
\[
\T(V^\ast) := \bigoplus_{n=0}^\infty \otimes^n V^\ast,
\]
where
\[
\otimes^n V^\ast = \{ \tau^n: \underbrace{V \times \cdots \times V}_{\text{$n$ times}} \longrightarrow \F \mid \mbox{$\tau^n$ is $n$-multilinear}\}.
\]
In particular, $\otimes^0 V^\ast := \F$ and $\otimes^1 V^\ast := V^\ast$. $\T(V^\ast)$. Then $\T(V^\ast)$ is a graded associative unital algebra, where the product $\otimes: \otimes^n V^\ast \times \otimes^m V^\ast \to \otimes^{n+m} V^\ast$ is defined as
\begin{equation} \label{eq:tensorprod}
(\tau_1^n \otimes \tau_2^m) (v_1, \ldots, v_{n+m}) := \tau_1^n(v_1, \ldots, v_n) \cdot \tau_2^m(v_{n+1}, \ldots, v_{n+m}).
\end{equation}
By convention, the multiplication with elements of $\otimes^0 V^\ast = \F$ is the scalar multiplication, so that $1 \in \F$ is the unit of $\T(V^\ast)$. Observe that $\T(V^\ast)$ is non-commutative.

There is a linear action of $S_n$, the permutation group of $n$ elements, on $\otimes^n V^\ast$ given by
\begin{equation} \label{eq:tensorperm}
(P_\sigma \tau^n)(v_1, \ldots, v_n) := \tau^n(v_{\sigma^{-1}(1)}, \ldots, v_{\sigma^{-1}(n)})
\end{equation}
for $\sigma \in S_n$ and $\tau^n \in \otimes^n V^\ast$. Indeed, the identity $P_{\sigma_1}(P_{\sigma_2} \tau^n) = P_{\sigma_1\sigma_2} \tau^n$ is easily verified. We call a tensor $\tau^n \in \otimes^n V^\ast$ {\em symmetric}, if $P_\sigma \tau^n = \tau^n$ for all $\sigma \in S_n$, and we let
\[
\odot^n V^\ast := \{ \tau^n \in \otimes^n V^\ast \mid \mbox{$\tau^n$ is symmetric}\}
\]
the {\em $n$-fold symmetric power of $V^\ast$}. Evidently, $\odot^n V^\ast \subset \otimes^n V^\ast$ is a linear subspace.

A {\em unital subalgebra of $\T(V^\ast)$} is a linear subspace $\A \subset \T(V^\ast)$ containing $\F = \odot^0 V^\ast$ which is closed under tensor products, i.e. such that $\tau_1, \tau_2 \in \A$ implies that $\tau_1 \otimes \tau_2 \in \A$. We call such a subalgebra {\em graded} if
\[
\A = \bigoplus_{n=0}^\infty \A_n \quad \mbox{with $\A_n := \A \cap \otimes^n V^\ast$},
\]
and a graded subalgebra $\A \subset \T(V)$ is called {\em permutation invariant }if $\A_n$ is preserved by the action of $S_n$ on $\A_n \subset \otimes^n V^\ast$.

\begin{definition}
Let $\Ss \subset \T(V^\ast)$ be an arbitrary subset. The intersection of all permutation invariant unital subalgebras of $\T(V^\ast)$ containing $\Ss$ is called the {\em permutation invariant subalgebra generated by $\Ss$} and is denoted by $\A_{\text{perm}}(\Ss)$.
\end{definition}

Observe that $\A_{\text{perm}}(\Ss)$ is the smallest permutation invariant unital subalgebra of $\T(V^\ast)$ which contains $\Ss$.

\begin{example} Evidently, $\A_{\text{perm}}(\emptyset) = \F$.

To see another example, let $\tau^1 \in V^\ast$. If we let $\A_0 := \F$ and  $\A_n := \F (\underbrace{\tau^1 \otimes \cdots \otimes \tau^1}_{\text{$n$ times}})$ for $n \geq 1$, then $\A_{\text{perm}}(\tau^1) = \bigoplus_{n=0}^\infty \A_n$. In fact, $\A_{\text{perm}}(\tau^1)$ is even commutative and isomorphic to the algebra of polynomials over $\F$ in one variable. 
\end{example}

For $n \in \N$, we denote by $\mbox{\bf Part}(n)$ the collection of partitions ${\bf P} = \{P_1, \ldots, P_r\}$ of $\{1, \ldots, n\}$, that is, $\bigcup_k P_k = \{1, \ldots, n\}$, and these sets are pairwise disjoint.
We denote the number $r$ of sets in the partition by $|{\bf P}|$.

Given a partition ${\bf P} = \{P_1, \ldots, P_r\} \in \mbox{\bf Part}(n)$, we associate to it a bijective map
\begin{equation} \label{eq:pi-P}
\pi_{\bf P}:  \; \biguplus_{i \in \{1, \ldots, r\}} \left( \{ i \} \times \{1,\dots,n_i\} \right) \longrightarrow \{1, \ldots, n\},
\end{equation}
where $n_i := |P_i|$, such that $\pi_{\bf P} (\{ i \} \times \{1,\dots,n_i\}) =P_i$. This map is well defined, up to permutation of the elements in $P_i$.

$\mbox{\bf Part}(n)$ is partially ordered by the relation ${\bf P} \leq {\bf P}'$ if ${\bf P}$ is a subdivision of ${\bf P}'$. This ordering has the partition $\{ \{1\}, \ldots, \{n\}\}$ into singleton sets as its minimum and $\{ \{1, \ldots, n\} \}$ as its maximum.

Consider now a subset of $\T(V^\ast)$ of the form
\begin{equation} \label{eq:def-S}
\Ss := \{ \tau^1, \tau^2, \tau^3, \ldots\} \quad \mbox{containing one symmetric tensor $\tau^n \in \odot^n V^\ast$ for each $n \in \N$}.
\end{equation}
For a partition ${\bf P} \in \mbox{\bf Part}(n)$ with the associated map $\pi_{\bf P}$ from (\ref{eq:pi-P}) we define $\tau^{\bf P} \in \otimes^n V^\ast$ as
\begin{equation} \label{eq:cantens-P}
\tau^{\bf P}(v_1, \ldots, v_n) := \prod_{i=1}^r \tau^{n_i}(v_{\pi_{\bf P}(i, 1)}, \ldots, v_{\pi_{\bf P}(i, n_i)}).
\end{equation}

Observe that this definition is independent of the choice of the bijection $\pi_{\bf P}$, since $\tau^{n_i}$ is symmetric.

\begin{example} \label{ex:partition}
\begin{enumerate}
\item
If ${\bf P} = \{ \{1, \ldots, n\} \}$ is the trivial partition, then
\[
\tau^{\bf P} = \tau^n.
\]
\item
If ${\bf P} = \{ \{1\}, \ldots, \{n\}\}$ is the partition into singletons, then
\[
\tau^{\bf P}(v_1, \ldots, v_n) = \tau^1(v_1) \cdots \tau^1(v_n).
\]
\item
To give a concrete example, let $n = 5$ and ${\bf P} = \{\{1, 3\}, \{2, 5\}, \{4\}\}$. Then
\[
\tau^{\bf P}(v_1, \ldots, v_5) = \tau^2(v_1, v_3) \cdot \tau^2(v_2, v_5) \cdot \tau^1(v_4).
\]
\end{enumerate}
\end{example}

We can now present the main result of this section.

\begin{proposition} \label{prop:generated-algebra}
Let $\Ss \subset \T(V^\ast)$ be given as in (\ref{eq:def-S}). Then the permutation invariant subalgebra generated by $\Ss$ equals
\begin{equation} \label{eq:ApermS}
\A_{\text{\em perm}}(\Ss) = \F \oplus \bigoplus_{n=1}^\infty \span \left\{ \tau^{\bf P} \mid {\bf P} \in \mbox{\bf Part}(n)\right\}.
\end{equation}
\end{proposition}

\begin{proof}
Let us denote the right hand side of (\ref{eq:ApermS}) by $\A_{\text{perm}}'(\Ss)$, so that we wish to show that $\A_{\text{perm}}(\Ss) = \A_{\text{perm}}'(\Ss)$.

By Example \ref{ex:partition}.1, $\tau^n \in \A_{\text{perm}}'(\Ss)$ for all $n \in \N$, whence $\Ss \subset \A_{\text{perm}}'(\Ss)$. Furthermore, by (\ref{eq:cantens-P}) we have 
\[
\tau^{\bf P} \otimes \tau^{{\bf P}'} = \tau^{{\bf P} \cup {\bf P}'},
\]
where ${\bf P} \cup {\bf P}' \in {\bf Part}(n+m)$ is the partition of $\{1, \ldots, n+m\}$ obtained by regarding ${\bf P} \in {\bf Part}(n)$ and ${\bf P}' \in {\bf Part}(m)$ as partitions of $\{1, \ldots, n\}$ and $\{n+1, \ldots, n+m\}$, respectively. Moreover, if $\sigma \in S_n$ is a permutation and ${\bf P} = \{P_1, \ldots, P_r\}$ a partition, then the definition in (\ref{eq:cantens-P}) implies that
\[
P_\sigma (\tau^{\bf P}) = \tau^{\sigma^{-1}{\bf P}}, \quad \mbox{where $\sigma^{-1}(\{P_1, \ldots, P_r\}) := \{\sigma^{-1}P_1, \ldots, \sigma^{-1}P_r\}$}.
\]

That is, $\A_{\text{perm}}'(\Ss) \subset \T(V^\ast)$ is a permutation invariant unital subalgebra of $\T(V^\ast)$ containg $\Ss$, whence $\A_{\text{perm}}(\Ss) \subset \A_{\text{perm}}'(\Ss)$.

For the converse, observe that for a partition ${\bf P} = \{P_1, \ldots, P_r\} \in {\bf Part}(n)$, we may -- after applying a permutation of $\{1, \ldots, n\}$ -- assume that
\[
P_1 = \{1, \ldots, k_1\}, P_2 = \{k_1+1, \ldots, k_1+k_2\}, \ldots, P_r = \{n-k_r+1, \ldots, n\},
\]
with $k_i = |P_i|$, and in this case, (\ref{eq:tensorprod}) and (\ref{eq:cantens-P}) implies that
\[
\tau^{\bf P} = (\tau^{k_1}) \otimes (\tau^{k_2}) \otimes \cdots \otimes (\tau^{k_r}) \in \A_{\text{perm}}(\Ss),
\]
so that any permutation invariant subalgebra containing $\Ss$ also must contain $\tau^{\bf P}$ for all partitions, and this shows that $\A_{\text{perm}}'(\Ss) \subset \A_{\text{perm}}(\Ss)$.
\end{proof}

\subsection{Tensor fields} \label{sec:tens}

Recall that a {\em (covariant) $n$-tensor field\footnote{Since we do not consider non-covariant $n$-tensor fields in this paper, we shall suppress the attribute {\em covariant}.} $\Psi$ on a manifold $M$}
is a collection of $n$-multilinear forms $\Psi_p$ on $T_pM$ for all $p \in M$ such that for continuous vector fields $X^1, \ldots, X^n$ on $M$ the function
\[
p \longmapsto \Psi_p(X^1_p, \ldots, X^n_p)
\]
is continuous. This notion can also be adapted to the case where $M$ has a weaker structre than that of a manifold. The examples we have in mind are the subsets $\Pp^r(\Om) \subset \Mm^r(\Om)$ of $\Ss^r(\Om)$ for an arbitrary measurable space $\Om$ and $r \in (0,1]$, which fail to be manifolds. Nevertheless, there is a natural notion of {\em tangent cone at $\mu_r$} of these sets which is the collection of the derivatives of all curves in $\Mm^r(\Om)$ (in $\Pp^r(\Om)$, respectively) through $\mu_r$. These cones were determined in \cite[Proposition 2.1]{AJLS2015} as
\[
T_{\mu^r}\Mm^r(\Om) = \Ss^r(\Om; \mu) \quad \mbox{and} \quad T_{\mu^r}\Pp^r(\Om) = \Ss^r_0(\Om; \mu).
\]
with $\mu \in \Mm(\Om)$ ($\mu \in \Pp(\Om)$, respectively). Then in analogy to the notion for general manifolds, we can now define the notion of $n$-tensor field on $\Mm^r(\Om)$ and $\Pp^r(\Om)$ as follows.

\begin{definition} \label{def:tensor-MMr}
Let $\Om$ be a measurable space and $r \in (0,1]$. A {\em vector field on $\Mm^r(\Om)$ }is a continuous map $X: \Mm^r(\Om) \to \Ss^r(\Om)$ such that $X_{\mu^r} \in T_{\mu^r}\Mm^r(\Om)$ for all $\mu^r \in \Mm^r(\Om)$. The notion of a vector field on $\Pp^r(\Om)$ is defined analogously.

A {\em (covariant) $n$-tensor field on $\Mm^r(\Om)$ }is a collection of $n$-multilinear forms $\Psi_{\mu^r}$ on $T_{\mu^r}\Mm^r(\Om)$ for all $\mu^r \in \Mm^r(\Om)$ such that for continuous vector fields $X^1, \ldots, X^n$ on $\Mm^r(\Om)$ the function
\[
\mu^r \longmapsto \Psi_{\mu^r}(X^1_{\mu^r}, \ldots, X^n_{\mu^r})
\]
is continuous. The notion of vector fields and $n$-tensor fields on $\Pp^r(\Om)$ is defined analogously.
\end{definition}

If $\Psi, \Psi'$ are tensor fields of degree $n$ and $m$, respectively, and $\sigma \in S_n$ is a permutation, then the pointwise tensor product $\Psi \otimes \Psi'$ and the permutation $P_\sigma \Psi$ defined in (\ref{eq:tensorprod}) and (\ref{eq:tensorperm}) are tensor fields of degree $n+m$ and $n$, respectively. Moreover, for a differentiable map $f: N \to M$ the {\em pull-back of $\Psi$ under $f$ }is the tensor field on $N$ defined by
\begin{equation} \label{eq:tensor-pullback}
f^\ast \Psi(v_1, \ldots, v_n) := \Psi(df(v_1), \ldots, df(v_n)).
\end{equation}
Evidently, we have
\begin{equation} \label{eq:pullback-homom}
f^\ast(\Psi \otimes \Psi') = (f^\ast \Psi) \otimes (f^\ast \Psi') \qquad \mbox{and} \qquad P_\sigma(f^\ast \Psi) = f^\ast (P_\sigma \Psi).
\end{equation}

For instance, if $(M, \Om, \pb)$ is a $k$-integrable parametrized measure model, then by (\ref{eq:formal-derivative}), $d_\xi \pb^{1/k}(v) \in \Ss^{1/k}(\Om; \mu) = T_{\pb^{1/k}(\xi)} \Mm^{1/k}(\Om)$, so that for any $n$-tensor field $\Psi$ on $\Mm^{1/k}(\Om)$ the pull-back
\[
(\pb^{1/k})^\ast \Psi(v_1, \ldots, v_n) := \Psi(d \pb^{1/k}(v_1), \ldots, d \pb^{1/k}(v_n))
\]
is well defined. The same holds if $\pb: M \to \Pp(\Om)$ is a statistical model and $\Psi$ is an $n$-tensor field on $\Pp^{1/k}(\Om)$. Moreover, (\ref{eq:pullback-homom}) holds in this context as well when replacing $f$ by $\pb^{1/k}$.

\begin{definition} \label{def:tau-n-Omr}
Let $\Om$ be a measurable space, $n \in \N$ an integer and $0 < r \leq 1/n$. Then {\em canonical $n$-tensor field on $\Ss^r(\Om)$} is defined as the pull-back
\begin{equation} \label{eq:tau-n-Omr}
\tau^n_{\Om';r} := (\pi^{1/nr})^\ast L^n_\Om
\end{equation}
with the symmetric $n$-tensor $L^n_\Om$ on $\Ss^{1/n}(\Om)$ defined in (\ref{eq:canonical-nform}). The definition of the pullback in (\ref{eq:tensor-pullback}) and the formula for the Fr\'echet-derivative of $\pi^{1/nr}$ in (\ref{eq:deriv-pik}) now imply by a straightforward calculation that
\begin{equation} \label{def:canonical-tensor-r}
(\tau^n_{\Om;r})_{\mu_r}(\nu_1, \ldots, \nu_n) := \begin{cases} \displaystyle{\frac1{r^n} \int_\Om d(\nu_1 \cdot \ldots \cdot \nu_n \cdot |\mu_r|^{1/r - n})} & \mbox{if $r < 1/n$},\\[5mm]
L^n_\Om(\nu_1, \ldots, \nu_n) & \mbox{if $r = 1/n$}, \end{cases}
\end{equation}
where $\mu_r \in \Ss^r(\Om)$ and $\nu_i \in \Ss^r(\Om) = T_{\mu_r}\Ss^r(\Om)$.

Furthermore, if $(M, \Om, \pb)$ is a $k$-integrable parametrized measure model, $k := 1/r \geq n$, then we define the {\em canonical $n$-tensor field of $(M, \Om, \pb)$} as the pull-back
\begin{equation} \label{eq:tau-n-p}
\tau^n_{(M, \Om, \pb)} := (\pb^{1/k})^\ast \tau^n_{\Om;r} = (\pb^{1/n}) L^n_\Om.
\end{equation}
\end{definition}
In this case, (\ref{eq:formal-derivative}) implies that for $v_1, \ldots, v_n \in T_\xi M$
\begin{equation} \label{eq:tau-n-plog}
\tau^n_{(M, \Om, \pb)}(v_1, \ldots, v_n) = \int_\Om \p_{v_1} \log \pb(\xi) \cdots \p_{v_n} \log \pb(\xi)\; d\pb(\xi).
\end{equation}

\begin{example} \begin{enumerate}
\item
The canonical $1$-tensor of $(M, \Om, \pb)$ is given as
\[
\left(\tau^1_{(M, \Om, \pb)}\right)_\mu(v) = \int_\Om \p_{v_1} \log \pb(\xi)\; d\pb(\xi) = \p_v \|\pb(\xi)\|.
\]
Thus, on a statistical model (i.e., if $\pb(\xi) \in \Pp(\Om)$ for all $\xi$) $\tau^1_{(M, \Om, \pb)} \equiv 0$.

\item
The canonical $2$-tensor $\tau^2_{(M, \Om, \pb)}$ is called the {\em Fisher metric }of the model and is often simply denoted by $\g$. It is defined only if the model is $2$-integrable.
\item
The canonical $3$-tensor $\tau^3_{(M, \Om, \pb)}$ is called the {\em Amari-Chentsov tensor }of the model. It is often simply denoted by $\T$ and is defined only if the model is $3$-integrable.
\end{enumerate}
\end{example}

\section{Congruent families of tensor fields} \label{sec:congruent-general}

The question we wish to address in this section is to characterize families of $n$-tensor fields on $\Mm^r(\Om)$ (on $\Pp^r(\Om)$, respectively) for measurable spaces $\Om$ which are unchanged under congruent Markov morphisms. 

First of all, we need to clarify what is meant by this. The problem we have is that a given Markov kernel $K: \Om \to \Pp(\Om)$ induces the bounded linear Markov morphism $K_\ast: \Ss(\Om) \to \Ss(\Om')$ which maps $\Pp(\Om)$ and $\Mm(\Om)$ to $\Pp(\Om')$ and $\Mm(\Om')$, respectively, there is no induced differentiable map from $\Pp^r(\Om)$ and $\Mm^r(\Om)$ to $\Pp^r(\Om')$ and $\Mm^r(\Om')$, respectively, if $r < 1$. The best we can do is to make the following definition.

\begin{definition} \label{def:K*r}
Let $K: \Om \to \Pp(\Om')$ be a Markov kernel with the associated Markov morphism $K_\ast: \Ss(\Om) \to \Ss(\Om')$ from (\ref{eq:Markov-linear}). For $r \in (0,1]$ we define
\begin{equation} \label{eq:def-Kr}
K_r: \Ss^r(\Om) \to \Ss^r(\Om'), \qquad K_r := \pi^r K_\ast \pi^{1/r},
\end{equation}
which maps $\Pp^r(\Om)$ and $\Mm^r(\Om)$ to $\Pp^r(\Om')$ and $\Mm^r(\Om')$, respectively.
\end{definition}

Since $r \leq 1$, it follows that $\pi^{1/r}$ is a Fr\'echet-$C^1$-map and $K_\ast$ is linear. However, $\pi^r$ is continuous but not differentiable for $r < 1$, whence the same holds for $K_r$.

Nevertheless, let us for the moment pretend that $K_r$ was differentiable. Then, when rewriting (\ref{eq:def-Kr}) as $\pi^{1/r} K_r = K_\ast \pi^{1/r}$, the chain rule and (\ref{eq:formal-derivative}) would imply that
\begin{equation} \label{eq:diffK*r-1}
|K_r\mu_r|^{1/r-1} \cdot (d_{\mu_r} K_r \nu_r) = K_\ast (|\mu_r|^{1/r-1} \cdot \nu_r)
\end{equation}
for all $\mu_r, \nu_r \in \Ss^r(\Om)$. 

On the other hand, as $K_r$ maps $\Mm^r(\Om)$ to $\Mm^r(\Om')$, its differential at $\mu^r \in \Mm^r(\Om)$ for $\mu \in \Mm(\Om)$ would restrict to a linear map
\[
d_{\mu^r} K_r: T_{\mu^r}\Mm^r(\Om) = \Ss^r(\Om, \mu) \longrightarrow T_{{\mu'}^r}\Mm^r(\Om') = \Ss^r(\Om, \mu'),
\]
where $\mu' := K_\ast\mu \in \Mm(\Om')$. This together with (\ref{eq:diffK*r-1}) implies that the restriction of $d_{\mu_r} K_r$ to $\Ss^r(\Om, \mu)$ must be given as

\begin{equation} \label{eq:K-mu-r}
d_{\mu^r} K_r: \Ss^r(\Om, \mu) \longrightarrow \Ss^r(\Om', \mu'), \qquad d_{\mu^r} K_r(\phi \mu^r) = \dfrac{d\{K_*(\phi \mu)\}}{d\mu'}\; {\mu'}^r.
\end{equation}

Indeed, by \cite[Theorem 3.3]{AJLS2015}, (\ref{eq:K-mu-r}) defines a bounded linear map $d_{\mu^r} K_r$. In fact, it is shown in that reference that
\[
\left\|d_{\mu^r} K_r(\phi \mu^r)\right\|_{\Ss^r(\Om', \mu')} = \left\|\dfrac{d\{K_*(\phi \mu)\}}{d\mu'}\right\|_{L^{1/r}(\Om', \mu')} \leq \|\phi\|_{L^{1/r}(\Om, \mu)} = \|\phi \mu^r\|_{\Ss^r(\Om, \mu)}.
\]

\begin{definition} \label{def:formal-deriv}
For $\mu \in \Mm(\Om)$, the bounded linear map (\ref{eq:K-mu-r}) is called the {\em formal derivative of $K_r$ at $\mu$}.
\end{definition}

If $(M, \Om, \pb)$ is a $k$-integrable parametrized measure model, then so is $(M, \Om', \pb')$ with $\pb' := K_\ast \pb$ by Proposition \ref{prop:formal-p1/k}. In this case, we may also write
\begin{equation} \label{eq:p'1/k}
{\pb'}^{1/k} = K_{1/k} \pb^{1/k}.
\end{equation}

\begin{proposition} \label{prop:formal-deriv-p1/k}
The formal derivative of $K_r$ defined in (\ref{eq:K-mu-r}) satisfies the identity
\[
d_\xi {\pb'}^{1/k} = (d_{\pb(\xi)^{1/k}}K_{1/k}) (d_\xi \pb^{1/k})
\]
for all $\xi \in M$ which may be regarded as the chain rule applied to the derivative of (\ref{eq:p'1/k}).
\end{proposition}

\begin{proof} For $v \in T_\xi M$, $\xi \in M$ we calculate
\begin{eqnarray*}
(d_{\pb(\xi)^{1/k}}K_{1/k}) (d_\xi \pb^{1/k}(v)) & \stackrel{(\ref{eq:formal-derivative})}= & \dfrac1k (d_{\pb^{1/k}(\xi)}K_{1/k}) (\p_v \log \pb(\xi) \pb(\xi)^{1/k})\\
& \stackrel{(\ref{eq:K-mu-r})} = & \dfrac1k \dfrac{d\{K_*(\p_v \log \pb(\xi) \pb(\xi))\}}{d\{\pb'(\xi)\}} \pb'(\xi)^{1/k}\\
& = & \dfrac1k \dfrac{d\{K_*(d_\xi \pb(v))\}}{d\{\pb'(\xi)\}} \pb'(\xi)^{1/k}\\
& = & \dfrac1k \dfrac{d\{d_\xi \pb'(v)\}}{d\{\pb'(\xi)\}} \pb'(\xi)^{1/k}\\
& = & \dfrac1k \p_v \log \pb'(\xi)\; \pb'(\xi)^{1/k} \stackrel{(\ref{eq:formal-derivative})}= d_\xi {\pb'}^{1/k}(v),
\end{eqnarray*}
which shows the assertion.
\end{proof}

Our definition of formal derivatives is just strong enough to define the pullback of tensor fields on the space of probability measures in analogy to (\ref{eq:tensor-pullback}).

\begin{definition}[Pullback of tensors by a Markov morphism]
Let $K: \Om \to \Pp(\Om')$ be a Markov kernel, and let $\Psi^n$ be an $n$-tensor field on $\Mm^r(\Om')$ (on $\Pp^r(\Om')$, respectively), cf.\ Definition \ref{def:tensor-MMr}.
Then the {\em pull-back tensor under $K$} is defined as the covariant $n$-tensor $K_r^\ast\Psi^n$ on $\Mm^r(\Om)$ (on $\Pp^r(\Om)$, respectively) given as
\[
K_r^*\Psi^n (V_1, \ldots, V_n) := \Psi^n(d K_r(V_1), \ldots, dK_r(V_n))
\]
with the formal derivative $dK_r$ from (\ref{eq:K-mu-r}).
\end{definition}

Evidently, $K_r^*\Psi^n$ is again a covariant $n$-tensor on $\Pp^r(\Om)$ and $\Mm^r(\Om)$, respectively, since $dK_r$ is continuous. Moreover, Proposition \ref{prop:formal-deriv-p1/k} implies that for a parametrized measure model $(M, \Om, \pb)$ and the induced model $(M, \Om', \pb')$ with $\pb' = K_* \pb$ we have the identity
\begin{equation} \label{eq:p-Kr*}
{\pb'}^\ast\Psi^n = \pb^*K_r^\ast\Psi^n
\end{equation}
for any covariant $n$-tensor field $\Psi^n$ on $\Pp^r(\Om)$ or $\Mm^r(\Om)$, respectively.

With this, we can now give a definition of congruent families of tensor fields.

\begin{definition}[Congruent families of tensors]  \label{def:algebr-gen}
Let $r \in (0, 1]$, and let $(\Theta_{\Om;r}^n)$ be a collection of covariant $n$-tensors on $\Pp^r(\Om)$ (on $\Mm^r(\Om)$, respectively) for each measurable space $\Om$.

This collection is said to be a {\em congruent family of $n$-tensors of regularity $r$ }if for any congruent Markov kernel $K: \Om \to \Om'$ we have
\[
K_r^*\Theta_{\Om';r}^n = \Theta_{\Om;r}^n.
\]
\end{definition}

The following gives an important example of such families. 

\begin{proposition} \label{prop:Ln-cong}
The restriction of the canonical $n$-tensors $L^n_\Om$ (\ref{eq:canonical-nform}) to $\Pp^{1/n}(\Om)$ and $\Mm^{1/n}(\Om)$, respectively, yield a congruent family of $n$-tensors. Likewise, then canonical $n$-tensors $(\tau^n_{\Om;r})$ on $\Pp^r(\Om)$ and $\Mm^r(\Om)$, respectively, with $r \leq 1/n$ yield congruent families of $n$-tensors.
\end{proposition}

\begin{proof}
Let $K: \Om \to \Pp(\Om')$ be a Markov kernel which is congruent w.r.t. the statistic $\kappa: \Om' \to \Om$ (cf. Definition \ref{def:CongMarkov}). For $\mu \in \Mm(\Om)$ let $\mu' := K_\ast \mu \in \Mm(\Om')$, so that $\kappa_\ast \mu' = \kappa_\ast K_\ast \mu = \mu$. Let $\nu_{1/n}^i = \phi_i {\mu}^{1/n} \in T_{\mu^r}\Mm^{1/n}(\Om) = \Ss^{1/n}(\Om, \mu')$, with $\phi_i \in L^{1/n}(\Om, \mu)$, $i=1, \ldots, n$, and define $\phi_i' \in L^{1/n}(\Om', \mu')$ by
\[
K_\ast(\phi_i \mu) = \phi' \mu'.
\]
By the $\kappa$-congruency of $K$, this implies that
\[
\phi_i \mu = \kappa_\ast K_\ast(\phi_i \mu) = \kappa_*(\phi_i' \mu') = (\kappa^\ast \phi_i') \kappa_\ast \mu' = (\kappa^\ast \phi_i') \kappa_\ast K_\ast \mu = (\kappa^\ast \phi_i') \mu,
\]
where $\kappa^\ast \phi(\cdot) := \phi(\kappa(\cdot))$, so that
\[
\kappa^\ast \phi'_i = \phi_i.
\]
Then
\begin{eqnarray*}
(K_{1/n}^*L_{\Om'}^n)_{\mu^{1/n}}(\nu_{1/n}^1, \ldots \nu_{1/n}^n) & = & L_{\Om'}^n \left((d_{\mu^{1/n}}K_{1/n}) (\phi_1 \mu^{1/n}), \ldots, (d_{\mu^{1/n}}K_{1/n}) (\phi_n \mu^{1/n})\right)\\
& \stackrel{(\ref{eq:K-mu-r})}= & L_{\Om'}^n (\phi_1' {\mu'}^{1/n}, \ldots, \phi_n' {\mu'}^{1/n})\\
& \stackrel{(\ref{eq:canonical-nform})}= & n^n \int_{\Om'} \phi_1' \cdots \phi_n' d\mu'\\
& = & n^n \int_\Om \kappa^\ast(\phi_1' \cdots \phi_n') d(\kappa_\ast \mu')\\
& = & n^n \int_\Om \phi_1 \cdots \phi_n d\mu = L^n_\Om(\nu_{1/n}^1, \ldots, \nu_{1/n}^n).
\end{eqnarray*}
This shows that $(L^n_\Om)$ is a congruent family of $n$-tensors. For $r \leq 1/n$, observe that by (\ref{eq:def-Kr}) we have
\[
K_r = \pi^{rn} K_{1/n} \pi^{1/rn} \Longrightarrow K_r^\ast = (\pi^{1/rn})^\ast K_{1/n}^\ast (\pi^{rn})^\ast
\]
and hence,
\[
K_r^\ast \tau^n_{\Om';r} \stackrel{(\ref{eq:tau-n-Omr})} = (\pi^{1/rn})^\ast K_{1/n}^\ast (\pi^{rn})^\ast (\pi^{1/rn})^\ast L^n_{\Om'} = (\pi^{1/rn})^\ast K_{1/n}^\ast L^n_{\Om'} = (\pi^{1/rn})^\ast L^n_\Om \stackrel{(\ref{eq:tau-n-Omr})} = \tau^n_{\Om;r},
\]
showing the congruency of the family $\tau^n_{\Om;r}$ as well.
\end{proof}

By (\ref{eq:pullback-homom}) and Definition \ref{def:algebr-gen}, it follows that tensor products and permutations of congruent families of tensors yield again such families. Moreover, since
\[
\|K_r\ast(\mu_r)\|_{\Ss^r(\Om')} = \|K_\ast \mu_r^{1/r}\|_{\Ss(\Om')} \stackrel{(\ref{eq:Markov-preserve})}= \|\mu_r^{1/r}\|_{\Ss(\Om)},
\]
multiplying a congruent family with a continuous function depending only on $\|\mu_r^{1/r}\|_{\Ss(\Om)} = \|\mu_r^{1/r}\|$ yields again a congruent family of tensors. Therefore, defining for a partition ${\bf P} \in {\bf Part}(n)$ with the associated map $\pi_{\bf P}$ from (\ref{eq:pi-P}) the tensor $\tau^{\bf P} \in \otimes^n V^\ast$ as
\begin{equation} \label{eq:cantens-P-Om}
(\tau_{\Om;r}^{\bf P})_{\mu_r}(v_1, \ldots, v_n) := \prod_{i=1}^r (\tau_{\Om;r}^{n_i})_{\mu_r}(v_{\pi_{\bf P}(i, 1)}, \ldots, v_{\pi_{\bf P}(i, n_i)}),
\end{equation}
this together with Proposition \ref{prop:generated-algebra} yields the following.

\begin{proposition} \label{prop:alg-gen}
For $r \in (0,1]$,
\begin{equation} \label{eq:congruent-form-r}
(\tilde \Theta^n_{\Om;r})_{\mu_r} = \sum_{\bf P} a_{\bf P}(\|\mu_r^{1/r}\|) (\tau^{\bf P}_{\Om;r})_{\mu_r},
\end{equation}
is a congruent family of $n$-tensor fields on $\Mm^r(\Om)$, where the sum is taken over all partitions ${\bf P} = \{P_1, \ldots, P_l\} \in {\bf Part}(n)$ with $|P_i| \leq 1/r$ for all $i$, and where $a_{\bf P}: (0, \infty) \to \R$ are continuous functions. Furthermore, 
\begin{equation} \label{eq:congruent-form-r-P}
\Theta^n_{\Om;r} = \sum_{\bf P} c_{\bf P} \tau^{\bf P}_{\Om;r},
\end{equation}
is a congruent family of $n$-tensor fields on $\Pp^r(\Om)$, where the sum is taken over all partitions ${\bf P} = \{P_1, \ldots, P_l\} \in {\bf Part}(n)$ with $1 < |P_i| \leq 1/r$ for all $i$, and where the $c_{\bf P} \in \R$ are constants.
\end{proposition}

In the light of Proposition \ref{prop:generated-algebra}, it is reasonable to use the following terminology.

\begin{definition} \label{def:alg-gen}
The congruent families of $n$-tensors on $\Mm^r(\Om)$ and $\Pp^r(\Om)$ given in (\ref{eq:congruent-form-r}) and (\ref{eq:congruent-form-r-P}), respectively, are called the families which are {\em algebraically generated by the canonical tensors.}
\end{definition}

\section{Congruent families on finite sample spaces} \label{sec:congruent-finite}

In this section, we wish to apply our discussion of the previous sections to the case where the sample space $\Om$ is assumed to be a finite set, in which case it is denoted by $I$ rather than $\Om$.

The simplification of this case is due to the fact that in this case the spaces $\Ss^r(I)$ are finite dimensional. Indeed, we have
\begin{equation}\label{def-S(I)}
\begin{array}{lrllrl}
\Ss(I) & = & \left\{\mu = \sum_{i \in I} \mu_i \delta_i \mid \mu_i \in \R \right\},\\[2mm]
\Mm(I) & = & \left\{ \mu \in \Ss(I) \mid \mu_i \geq 0\right\}, &  \Pp(I) & = & \left\{ \mu \in \Ss(I) \mid \mu_i \geq 0, \sum_i \mu_i = 1\right\},\\[2mm]
\Mm_+(I) & := & \left\{ \mu \in \Ss(I)  \mid \mu_i > 0\right\}, & \Pp_+(I) & := & \left\{ \mu \in \Ss(I)  \mid \mu_i > 0, \sum_i \mu_i = 1\right\},
\end{array}
\end{equation}
where $\delta_i$ denotes the Dirac measure supported at $i \in I$. The norm on $\Ss(I)$ is then
\[
\left\|\sum_{i \in I} \mu_i \delta_i\right\| = \sum_{i \in I} |\mu_i|.
\]
The space $\Ss^r(I)$ is then given as
\begin{equation}\label{def-Sr(I)}
\begin{array}{lrllrl}
\Ss^r(I) & = & \left\{\mu_r = \sum_{i \in I} \mu_i \delta_i^r \mid \mu_i \in \R \right\},\\[2mm]
\Mm^r(I) & = & \left\{ \mu_r \in \Ss^r(I) \mid \mu_i \geq 0\right\}, &  \Pp(I) & = \left\{ \mu_r \in \Ss^r(I) \mid \mu_i \geq 0, \sum_i \mu_i^{1/r} = 1\right\},\\[2mm]
\Mm_+^r(I) & = & \left\{ \mu_r \in \Ss^r(I)  \mid \mu_i > 0\right\}, & \Pp_+(I) & = \left\{ \mu_r \in \Ss^r(I)  \mid \mu_i > 0, \sum_i \mu_i^{1/r} = 1\right\}.
\end{array}
\end{equation}
The sets $\Mm_+(I)$ and $\Pp_+(I) \subset \Ss(I)$ are manifolds of dimension $|I|$ and $|I|-1$, respectively. Indeed, $\Mm_+(I) \subset \Ss(I)$ is an open subset, whereas $\Pp_+(I)$ is an open subset of the affine hyperplane $\Ss_1(I)$, cf (\ref{eq:def-SsOm}). In particular, we have
\[
T_\mu \Pp_+(I) = \Ss_0(I) \qquad \mbox{and} \qquad T_\mu \Mm_+(I) = \Ss(I).
\]
The norm on $\Ss^r(I)$ is given as
\[
\left\|\sum_{i \in I} \mu_i \delta_i^r\right\|_{\Ss^r(I)} = \sum_{i \in I} |\mu_i|^{1/r},
\]
and the product $\cdot$ and the exponentiating map $\pi^k: \Ss^r(I) \to \Ss^{kr}(I)$ from above are given as
\begin{equation} \label{eq:prod-pi-I}
\left(\sum_i \mu_i \delta_i^r\right) \cdot \left(\sum_i \nu_i \delta_i^s\right) = \sum_i \mu_i \nu_i \delta_i^{r+s}, \qquad \pi^k \left(\sum_{i \in I} \mu_i \delta_i^r\right) = \sum_i \sign(\mu_i) |\mu_i|^k \delta_i^{kr}.
\end{equation}
Evidently, $\pi^k$ maps $\Mm_+^r(I)$ and $\Pp_+^r(I)$ to $\Mm_+^{kr}(I)$ and $\Pp_+^{kr}(I)$, respectively, and the restriction of $\pi^k$ to these sets is differentiable even if $k < 1$.

A Markov kernel between the finite sets $I = \{1, \ldots, m\}$ and $I' = \{1, \ldots, n\}$ is determined by the $(n \times m)$-Matrix $(K^i_{i'})_{i, i'}$ by
\[
K(\delta_i) = K^i = \sum_{i'} K_{i'}^i \delta_{i'},
\]
where $K_{i'}^i \geq 0$ and $\sum_{i'} K_{i'}^i = 1$ for all $i \in I$. Therefore, by linearity,
\[
K_*\left(\sum_i x_i \delta_i\right) = \sum_{i, i'} K_{i'}^i x_i \delta_{i'}.
\]
In particular, $K_\ast(\Pp_+(I)) \subset K_\ast(\Pp_+(I'))$ and $K_\ast(\Mm_+(I)) \subset K_\ast(\Mm_+(I'))$.

If $\kappa: I' \to I$ is a statistic between finite sets (cf. Example \ref{ex:congstat}) and if we denote the induced partition by $A_i := \kappa^{-1}(i) \subset I'$, then a Markov kernel $K: I \to \Pp(I')$ given by the matrix $(K^i_{i'})_{i, i'}$ as above is $\kappa$-congruent if and only if
\[
K^i_{i'} = 0 \qquad \mbox{whenever $i' \notin A_i$.}
\]

Since $(\delta_i^r)_{i \in I}$ is a basis of $\Ss^r(\Om)$, we can describe any $n$-tensor $\Psi^n$ on $\Ss^r(I)$ by defining for all multiindices $\vec i := (i_1, \ldots, i_n) \in I^n$ the component functions
\begin{equation} \label{eq:tens-comp}
\psi^{\vec i}(\mu_r) := (\Psi^n)_{\mu_r}(\delta_{i_1}^r, \ldots, \delta_{i_n}^r) =: (\Psi^n)_{\mu_r}(\delta_{\vec i}),
\end{equation}
which are real valued functions depending continuously on $\mu_r \in \Ss^r(I)$. Thus, by (\ref{def:canonical-tensor-r}), the component functions of the canonical tensor $\tau^n_{\Om;r}$ from (\ref{eq:tens-comp}) are given as
\begin{equation} \label{eq:can-tens-comp}
\mu_r = \sum_{i \in I} m_i \delta_i^r \in \Ss^r(I) \qquad \Longrightarrow \qquad \theta^{\vec i}_{I; r}(\mu_r) = \left\{ \begin{array}{cl} |m_i|^{1/r-n} & \mbox{if $\vec i = (i, \ldots, i)$,}\\ \\ 0 & \mbox{otherwise.} \end{array} \right.
\end{equation}

\begin{remark}
Observe that $\theta^{\vec i}_{I;r}$ is continuous on $\Mm_+^r(I)$ and hence $\tau^n_{I;r} = (\pi^{1/nr})^\ast L^n_I$ is well-defined on $\Mm^r_+(I)$ even if $ r > 1/n$, as on this set $m_i > 0$. This reflects the fact that the restriction $\pi^{1/nr}: \Mm^r_+(I) \to \Ss^{1/n}(I)$ is differentiable for {\em any }$r > 0$ by (\ref{eq:prod-pi-I}).

In particular, for $r = 1$, when restricting to $\Mm_+(I)$ or $\Pp_+(I)$, the canonical tensor fields
\[
(\tau^n_{I;1}) =: (\tau^n_I) \qquad \mbox{and} \qquad (\tau^{\bf P}_{I;1}) =: (\tau^{\bf P}_I)
\]
yield a congruent family of $n$-tensors on $\Mm_+(I)$ and $\Pp_+(I)$, respectively, as is verified as in the proof of Proposition \ref{prop:Ln-cong}. Therefore, the families of $n$-tensor fields
\begin{equation} \label{eq:congruent-form-rI}
(\tilde \Theta^n_I)_{\mu_r} = \sum_{{\bf P} \in {\bf Part}(n)} a_{\bf P}(\|\mu_r^{1/r}\|) (\tau^{\bf P}_I)_{\mu_r},
\end{equation}
on $\Mm_+(I)$ and
\begin{equation} \label{eq:congruent-form-r-PI}
\Theta^n_I = \sum_{{\bf P} \in {\bf Part}(n), |P_i| > 1} c_{\bf P} \tau^{\bf P}_I
\end{equation}
on $\Pp_+(I)$ are congruent, where in contrast to (\ref{eq:congruent-form-r}) and (\ref{eq:congruent-form-r-P}) we need not restrict the sum to partitions with $|P_i| \leq 1/r$ for all $i$. In analogy to Definition \ref{def:alg-gen} we call these the families of congruent tensors {\em algebraically generated by the canonical $n$-tensors $\{\tau_I^n\}$}.
\end{remark}

The main result of this section (Theorem \ref{thm:Chentsov-finite}) will be that (\ref{eq:congruent-form-rI}) and (\ref{eq:congruent-form-r-PI}) are the only families of congruent $n$-tensor fields which are defined on $\Mm_+(I)$ and $\Pp_+(I)$, respectively, for all {\em finite }sets $I$. In order to do this, we first deal with congruent families on $\Mm_+(I)$ only.

A multiindex $\vec i = (i_1, \ldots, i_n) \in I^n$ induces a partition ${\bf P}(\vec i)$ of the set $\{1, \ldots, n\}$ into the equivalence classes of the relation $k \sim l \Leftrightarrow i_k = i_l$. For instance, for $n = 6$ and pairwise distinct elements $i, j, k \in I$, the partition induced by $\vec i := (j, i, i, k, j, i)$ is
\[
{\bf P}(\vec i) = \{\{1, 5\}, \{2, 3, 6\}, \{4\} \}.
\]

Since the canonical $n$-tensors $\tau^n_I$ are symmetric by definition, it follows that for any partition ${\bf P} \in {\bf Part}(n)$ we have by (\ref{eq:cantens-P-Om})
\begin{equation} \label{eq:tau-P-order}
(\tau^{\bf P}_I)_\mu(\delta_{\vec i}) \neq 0 \Longleftrightarrow {\bf P} \leq {\bf P}(\vec i).
\end{equation}

\begin{lemma} \label{lem:cantens-P}
In (\ref{eq:congruent-form-rI}) and (\ref{eq:congruent-form-r-PI}) above, $a_{\bf P}: (0, \infty) \to \R$ and $c_{\bf P}$ are uniquely determined.
\end{lemma}

\begin{proof}
To show the first statement, let us assume that there are functions $a_{\bf P}: (0, \infty) \to \R$ such that
\begin{equation} \label{eq:aP-unique}
\sum_{{\bf P} \in {\bf Part}(n)} a_{\bf P}(\|\mu\|) (\tau^{\bf P}_{I;\mu}) = 0
\end{equation}
for all finite sets $I$ and $\mu \in \Mm_+(I)$, but there is a partition ${\bf P}_0$ with $a_{{\bf P}_0} \not \equiv 0$. In fact, we pick ${\bf P}_0$ to be minimal with this property, and choose a multiindex $\vec i \in I^n$ with ${\bf P}(\vec i) = {\bf P}_0$. Then
\begin{eqnarray*}
0 & = & \sum_{{\bf P} \in {\bf Part}(n)} a_{\bf P}(\|\mu\|) (\tau^{\bf P}_I)_\mu(\delta_{\vec i}) \stackrel{(\ref{eq:tau-P-order})}= \sum_{{\bf P} \leq {\bf P}_0} a_{\bf P}(\|\mu\|) (\tau^{\bf P}_I)_\mu(\delta_{\vec i})\\
& = & a_{{\bf P}_0}(\|\mu\|) (\tau^{{\bf P}_0}_I)_\mu(\delta_{\vec i}),
\end{eqnarray*}
where the last equation follows since $a_{\bf P} \equiv 0$ for ${\bf P} < {\bf P}_0$ by the minimality assumption on ${\bf P}_0$.

But $(\tau^{{\bf P}_0}_I)_\mu(\delta_{\vec i}) \neq 0$ again by (\ref{eq:tau-P-order}), since ${\bf P}(\vec i) = {\bf P}_0$, so that $a_{{\bf P}_0}(\|\mu\|) = 0$ for all $\mu$, contradicting $a_{{\bf P}_0} \not \equiv 0$.

Thus, (\ref{eq:aP-unique}) occurs only if $a_{\bf P} \equiv 0$ for all ${\bf P}$, showing the uniqueness of the functions $a_{\bf P}$ in (\ref{eq:congruent-form-rI}).

The uniqueness of the constants $c_{\bf P}$ in  (\ref{eq:congruent-form-r-PI}) follows similarly, but we have to account for the fact that $\delta_i \notin \Ss_0(I) = T_\mu \Pp_+(I)$. In order to get around this, let $I$ be a finite set and $J := \{0,1,2\} \times I$. For $i \in I$, we define
\[
V_i := 2 \delta_{(0,i)} - \delta_{(1,i)} - \delta_{(2, i)} \in \Ss_0(J),
\]
and for a multiindex $\vec i = (i_1, \ldots, i_n) \in I^n$ we let
\[
(\tau^{\bf P}_J)_\mu(V^{\vec i}) := (\tau^{\bf P}_J)_\mu (V_{i_1}, \ldots V_{i_n}).
\]
Multiplying this term out, we see that $(\tau^{\bf P}_J)_\mu(V^{\vec i})$ is a linear combination of terms of the form $(\tau^{\bf P}_J)_\mu (\delta_{(a_1, i_1)}, \ldots, \delta_{(a_n, i_n)})$, where $a_i \in \{0,1,2\}$. Thus, from (\ref{eq:tau-P-order}) we conclude that
\begin{equation} \label{eq:V_i}
(\tau^{\bf P}_J)_\mu(V^{\vec i}) \neq 0 \qquad \mbox{only if ${\bf P} \leq {\bf P}(\vec i)$}.
\end{equation}
Moreover, if ${\bf P}(\vec i) = \{P_1, \ldots, P_r\}$ with $|P_i| = k_i$, and $\mu_0 := 1/|J| \sum \delta_{(a,i)} \in \Pp_+(J)$, then 

\begin{eqnarray*}
(\tau^{k_i}_J)_{\mu_0} (V_i, \ldots, V_i) & \stackrel{(\ref{eq:can-tens-comp})}= & 2^{k_i} (\tau^{k_i}_J)_{\mu_0}(\delta_{(0,i)}, \ldots, \delta_{(0,i)})\\
&& \; + (-1)^{k_i} (\tau^{k_i}_J)_{\mu_0}(\delta_{(1,i)}, \ldots, \delta_{(1,i)}) + (-1)^{k_i} (\tau^{k_i}_J)_{\mu_0}(\delta_{(2,i)}, \ldots, \delta_{(2,i)})\\
& \stackrel{(\ref{eq:can-tens-comp})}= & (2^{k_i} + 2 (-1)^{k_i}) |J|^{k_i-1}.
\end{eqnarray*}

Thus, by (\ref{eq:cantens-P}) we have
\[
(\tau^{{\bf P}(\vec i)}_J)_{\mu_0}(V^{\vec i}) = \prod_{i=1}^r (\tau^{k_i}_J)_{\mu_0} (V_i, \ldots, V_i) = \prod_{i=1}^r (2^{k_i} + 2 (-1)^{k_i}) |J|^{k_i-1} = |J|^{n - r} \prod_{i=1}^r (2^{k_i} + 2 (-1)^{k_i}).
\]
In particular, since $2^{k_i} + 2 (-1)^{k_i} > 0$ for all $k_i \geq 2$ we conclude that
\begin{equation} \label{eq:V_i2}
(\tau^{{\bf P}(\vec i)}_J)_{\mu_0}(V^{\vec i}) \neq 0,
\end{equation}
as long as ${\bf P}(\vec i)$ does not contain singleton set.

With this, we can now proceed as in the previous case: assume that
\begin{equation} \label{eq:cP-unique}
\sum_{{\bf P} \in {\bf Part}(n), |P_i| \geq 2} c_{\bf P} \; \tau^{\bf P}_I = 0 \qquad \mbox{when restricted to $\Pp_+(I)$}
\end{equation}
for constants $c_{\bf P}$ which do not all vanish, and we let ${\bf P}_0$ be minimal with $c_{{\bf P}_0} \neq 0$. Let $\vec i = (i_1, \ldots, i_n) \in I^n$ be a multiindex with ${\bf P}(\vec i) = {\bf P}_0$, and let $J := \{0,1,2\} \times I$ be as above. Then
\begin{eqnarray*}
0 & = & \sum_{{\bf P} \in {\bf Part}(n), |P_i| \geq 2} c_{\bf P}\; (\tau^{\bf P}_J)_{\mu_0}(V^{\vec i}) \stackrel{(\ref{eq:V_i})}= \sum_{{\bf P} \leq {\bf P}_0, |P_i| \geq 2} c_{\bf P}\; (\tau^{\bf P}_J)_{\mu_0}(V^{\vec i})\\
& = & c_{{\bf P}_0}\; (\tau^{{\bf P}_0}_J)_{\mu_0}(V^{\vec i}),
\end{eqnarray*}
where the last equality follows by the assumption that ${\bf P}_0$ is minimal. But $(\tau^{{\bf P}_0}_J)_\mu(V^{\vec i}) \neq 0$ by (\ref{eq:V_i2}), whence $c_{{\bf P}_0} = 0$, contradicting the choice of ${\bf P}_0$.

This shows that (\ref{eq:cP-unique}) can happen only if all $c_{\bf P} = 0$, and this completes the proof.
\end{proof}

The main result of this section is the following.

\begin{theorem} \label{thm:Chentsov-finite} (Classification of congruent families of $n$-tensors)

The class of congruent families of $n$-tensors on $\Mm_+(I)$ and $\Pp_+(I)$, respectively, for finite sets $I$ is the class algebraically generated by the canonical $n$-tensors $\{\tau_I^n\}$. That is, these families are the ones given in (\ref{eq:congruent-form-rI}) and (\ref{eq:congruent-form-r-PI}), respectively.
\end{theorem}

The rest of this section will be devoted to its proof which is split up into several lemmas.

\begin{lemma} \label{lem:can-tens-P}
Let $\tau^{\bf P}_I $ be the canonical $n$-tensor from Definition \ref{eq:cantens-P-Om}, and define the center
\begin{equation} \label{eq:c-I}
c_I := \frac1{|I|} \sum_i \delta_i \in \Pp_+(I).
\end{equation}
Then for any $\lambda > 0$ we have
\begin{equation} \label{eq:eval-tau-P}
(\tau^{\bf P}_I)_{\lambda c_I}(\delta_{\vec i}) = \left\{ \begin{array}{cl} \displaystyle{\left(\frac{|I|}\lambda\right)^{n - |{\bf P}|}} & \mbox{if ${\bf P} \leq {\bf P}(\vec i)$,}\\ \\ 0 & \mbox{otherwise}. \end{array} \right.
\end{equation}
\end{lemma}

\begin{proof}
For $\mu = \lambda c_I$, $\lambda > 0$, the components $\mu_i$ of $\mu$ all equal $\mu_i = \lambda/|I|$, whence in this case we have for all multiindices $\vec i$ with ${\bf P} \leq {\bf P}(\vec i)$

\[
(\tau^{\bf P}_I)_{\lambda c_I}(\delta_{\vec i}) = \prod_{i=1}^r \theta_{I;\lambda c_I}^{i, \ldots, i} \stackrel{(\ref{eq:can-tens-comp})} = \prod_{i=1}^r \left( \frac{|I|}\lambda \right)^{k_i-1} = \left( \frac{|I|}\lambda \right)^{k_1+\ldots+k_r-r} = \left( \frac{|I|}\lambda \right)^{n-|{\bf P}|}\\
\]
showing (\ref{eq:eval-tau-P}). If ${\bf P} \not \leq {\bf P}(\vec i)$, the claim follows from (\ref{eq:tau-P-order}).
\end{proof}

Now let us suppose that $\{ \tilde{\Theta}^n_I\: : \; I\; \mbox{finite}\}$ is a congruent family of $n$-tensors on $\Mm_+(I)$, and define $\theta^{\vec i}_{I, \mu}$ as in (\ref{eq:tens-comp}) and $c_I \in \Pp_+(I)$ as in (\ref{eq:c-I}).

\begin{lemma} \label{lem:step1}
Let $\{ \tilde{\Theta}^n_I\: : \; I\; \mbox{finite}\}$ and $\theta^{\vec i}_{I, \mu}$ be as before, and let $\lambda > 0$. If $\vec i, \vec j \in I^n$ are multiindices with ${\bf P}(\vec i) = {\bf P}(\vec j)$, then
\[
\theta^{\vec i}_{I, \lambda c_I} = \theta^{\vec j}_{I, \lambda c_I}.
\]
\end{lemma}

\begin{proof}
If ${\bf P}(\vec i) = {\bf P}(\vec j)$, then there is a permutation $\sigma: I \to I$ such that $\sigma(i_k) = j_k$ for $k = 1, \ldots, n$. We define the congruent Markov kernel $K: I \to \Pp(I)$ by $K^i := \delta_{\sigma(i)}$. Then evidently, $K_*c_I = c_I$, and Definition \ref{def:algebr-gen} implies
\begin{eqnarray*}
\theta^{\vec i}_{I, \lambda c_I} & = & (\tilde{\Theta}^n_I)_{\lambda c_I} (\delta_{i_1}, \ldots, \delta_{i_n})\\
& = &
(\tilde{\Theta}^n_I)_{K_*(\lambda c_I)} (K_*\delta_{i_1}, \ldots, K_*\delta_{i_n})\\
& = & (\tilde{\Theta}^n_I)_{\lambda c_I} (\delta_{j_1}, \ldots, \delta_{j_n}) =  \theta^{\vec j}_{I, \lambda c_I},
\end{eqnarray*}
which shows the claim.
\end{proof}

By virtue of this lemma, we may define
\[
\theta^{\bf P}_{I, \lambda c_I} := \theta^{\vec i}_{I, \lambda c_I}, \qquad \mbox{where $\vec i \in I^n$ is a multiindex with ${\bf P}(\vec i) = {\bf P}$}.
\]

\begin{lemma} \label{lem:step2}
Let $\{ \tilde{\Theta}^n_I\: : \; I\; \mbox{finite}\}$ and $\theta^{\bf P}_{I, \lambda c_I}$ be as before, and suppose that ${\bf P}_0 \in \mbox{\bf Part}(n)$ is a partition such that
\begin{equation} \label{eq:tensor-clustered}
\theta^{\bf P}_{I, \lambda c_I} = 0 \qquad \mbox{for all ${\bf P} < {\bf P}_0$, $\lambda > 0$ and $I$}.
\end{equation}
Then there is a continuous function $f_{{\bf P}_0}: (0, \infty) \to \R$ such that
\begin{equation} \label{eq:step2}
\theta^{{\bf P}_0}_{I, \lambda c_I} =  f_{{\bf P}_0}(\lambda)\;|I|^{n-|{\bf P}_0|}.
\end{equation}
\end{lemma}

\begin{proof} Let $I, J$ be finite sets, and let $I' := I \times J$. We define the Markov kernel
\[
K: I \longrightarrow \Pp(I'), \qquad i \longmapsto \frac1{|J|} \sum_{j \in J} \delta_{(i,j)}
\]
which is congruent w.r.t. the canonical projecton $\kappa: I' \to I$. Then $K_*c_I = c_{I'}$ is easily verified. Moreover, if $\vec i = (i_1, \ldots, i_n) \in I^n$ is a multiindex with ${\bf P}(\vec i) = {\bf P}_0$, then
\begin{eqnarray*}
\theta^{{\bf P}_0}_{I, \lambda c_I} & = & (\tilde \Theta^n_I)_{\lambda c_I}(\delta_{i_1}, \ldots, \delta_{i_n})\\
& \stackrel{\text{Def. } \ref{def:algebr-gen}}= & (\tilde{\Theta}^n_{I'})_{K_*(\lambda c_I)}(K_*\delta_{i_1}, \ldots, K_*\delta_{i_n})\\
& = & (\tilde{\Theta}^n_{I'})_{\lambda c_{I'}}\left(\frac1{|J|} \sum_{j_1 \in J} \delta_{(i_1, j_1)}, \ldots, \frac1{|J|} \sum_{j_n \in J} \delta_{(i_n, j_n)}\right)\\
& = & \frac1{|J|^n} \sum_{(j_1, \ldots, j_n) \in J^n} \theta^{{\bf P}((i_1, j_1), \ldots, (i_n, j_n))}_{I', \lambda c_{I'}}.
\end{eqnarray*}
Observe that ${\bf P}((i_1, j_1), \ldots, (i_n, j_n)) \leq {\bf P}(\vec i) = {\bf P}_0$. If ${\bf P}((i_1, j_1), \ldots, (i_n, j_n)) < {\bf P}_0$, then $\theta^{{\bf P}((i_1, j_1), \ldots, (i_n, j_n))}_{I', \lambda c_{I'}} = 0$ by (\ref{eq:tensor-clustered}).

Moreover, there are $|J|^{|{\bf P}_0|}$ multiindices $(j_1, \ldots, j_n) \in J^n$ for which
\linebreak
${\bf P}((i_1, j_1), \ldots, (i_n, j_n)) = {\bf P}_0$, and since for all of these $\theta^{{\bf P}((i_1, j_1), \ldots, (i_n, j_n))}_{I', \lambda c_{I'}} = \theta^{{\bf P}_0}_{I', \lambda c_{I'}}$, we obtain
\[
\theta^{{\bf P}_0}_{I, \lambda c_I} = \frac1{|J|^n} \sum_{(j_1, \ldots, j_n) \in J^n} \theta^{{\bf P}((i_1, j_1), \ldots, (i_n, j_n))}_{I', \lambda c_{I'}} = \frac{|J|^{|{\bf P}_0|}}{|J|^n}\; \theta^{{\bf P}_0}_{I', \lambda c_{I'}} = \frac1{|J|^{n -|{\bf P}_0|}}\; \theta^{{\bf P}_0}_{I', \lambda c_{I'}},
\]
and since $|I'| = |I|\; |J|$, it follows that
\[
\frac1{|I|^{n - |{{\bf P}_0}|}} \theta^{{\bf P}_0}_{I, \lambda c_I} = \frac1{|I|^{n - |{{\bf P}_0}|}} \left( \frac1{|J|^{n-|{\bf P}_0|}}\; \theta^{{\bf P}_0}_{I', \lambda c_{I'}}\right) = \frac1{|I'|^{n - |{{\bf P}_0}|}} \theta^{{\bf P}_0}_{I', \lambda c_{I'}}.
\]
Interchanging the roles of $I$ and $J$ in the previous arguments, we also get
\[
\frac1{|J|^{n - |{\bf P}_0|}} \theta^{{\bf P}_0}_{J, \lambda c_J} = \frac1{|I'|^{n - |{{\bf P}_0}|}} \theta^{{\bf P}_0}_{I', \lambda c_{I'}} = \frac1{|I|^{n - |{\bf P}_0|}} \theta^{{\bf P}_0}_{I, \lambda c_I},
\]
whence $f_{{\bf P}_0}(\lambda) := \frac1{|I|^{n - |{\bf P}_0|}} \theta^{{\bf P}_0}_{I, \lambda c_I}$ is indeed independent of the choice of the finite set $I$.
\end{proof}

\begin{lemma} \label{lem:kill-center}
Let $\{ \tilde{\Theta}^n_I\: : \; I\; \mbox{finite}\}$ and $\lambda > 0$ be as before. Then there is a congruent family $\{ \tilde \Psi^n_I\: : \; I\; \mbox{finite}\}$ of the form (\ref{eq:congruent-form-rI}) such that
\[
(\tilde{\Theta}^n_I\ - \tilde \Psi^n_I)_{\lambda c_I} = 0 \qquad \mbox{for all finite sets $I$ and all $\lambda > 0$.}
\]
\end{lemma}

\begin{proof} For a congruent family of $n$-tensors $\{ \tilde{\Theta}^n_I\: : \; I\; \mbox{finite}\}$, we define
\[
N(\{\tilde{\Theta}^n_I\}) := \{ {\bf P} \in \mbox{\bf Part}(n) \; : \; (\tilde{\Theta}^n_I)_{\lambda c_I} (\delta_{\vec i}) = 0 \mbox{ whenever }{\bf P}(\vec i) \leq {\bf P}\}.
\]
If $N(\{\tilde{\Theta}^n_I\}) \subsetneq \mbox{\bf Part}(n)$, then let
\[
{\bf P}_0 = \{ P_1, \ldots, P_r\} \in \mbox{\bf Part}(n) \backslash N(\{\tilde{\Theta}^n_I\})
\]
be a minimal element, i.e., such that ${\bf P} \in N(\{\tilde{\Theta}^n_I\})$ for all ${\bf P} < {\bf P}_0$. In particular, for this partition (\ref{eq:tensor-clustered}) and hence (\ref{eq:step2}) holds. Let
\begin{equation} \label{eq:Theta'}
({\tilde{\Theta'}}^n_I)_\mu := (\tilde{\Theta}^n_I)_\mu - \|\mu\|^{n - |{\bf P}_0|} f_{{\bf P}_0}(\|\mu\|)\; (\tau^{{\bf P}_0}_I)_\mu
\end{equation}
with the function $f_{{\bf P}_0}$ from (\ref{eq:step2}). Then $\{{\tilde{\Theta'}}^n_I \; : \; I\; \mbox{finite}\}$ is again a family of $n$-tensors.

Let ${\bf P} \in N(\{\tilde{\Theta}^n_I\})$ and $\vec i$  be a multiindex with ${\bf P}(\vec i) \leq {\bf P}$. If $(\tau^{{\bf P}_0}_I)_{\lambda c_I}(\delta_{\vec i}) \neq 0$, then by Lemma \ref{lem:can-tens-P} we would have ${\bf P}_0 \leq {\bf P}(\vec i) \leq {\bf P} \in N(\{\tilde{\Theta}^n_I\})$ which would imply that ${\bf P}_0 \in N(\{\tilde{\Theta}^n_I\})$, contradicting the choice of ${\bf P}_0$.

Thus, $(\tau^{{\bf P}_0}_I)_{\lambda c_I}(\delta_{\vec i}) = 0$ and hence $({\tilde{\Theta'}}^n_I)_{\lambda c_I}(\delta_{\vec i}) = 0$ whenever ${\bf P}(\vec i) \leq {\bf P}$, showing that ${\bf P} \in N(\{{\tilde{\Theta'}}^n_I\})$.

Thus, what we have shown is that $N(\{\tilde{\Theta}^n_I\}) \subset N(\{{\tilde{\Theta'}}^n_I\})$.
On the other hand, if ${\bf P}(\vec i) = {\bf P}_0$, then again by Lemma \ref{lem:can-tens-P}
\[
(\tau^{{\bf P}_0}_I)_{\lambda c_I} (\delta_{\vec i}) = \left(\frac{|I|}\lambda\right)^{n - |{\bf P}_0|},
\]
and since $\|\lambda c_I\| = \lambda$, it follows that
\begin{eqnarray*}
({\tilde{\Theta'}}^n_I)_{\lambda c_I} (\delta_{\vec i}) & \stackrel{(\ref{eq:Theta'})}= & ({\tilde{\Theta}}^n_I)_{\lambda c_I} (\delta_{\vec i}) - \lambda^{n - |{\bf P}_0|} f_{{\bf P}_0}(\lambda) (\tau^{{\bf P}_0}_I)_{\lambda c_I} (\delta_{\vec i})\\
& = & \theta^{{\bf P}_0}_{I, \lambda c_I} - \lambda^{n - |{\bf P}_0|} f_{{\bf P}_0}(\lambda) \left(\frac{|I|}\lambda\right)^{n - |{\bf P}_0|}\\
& = & \theta^{{\bf P}_0}_{I, \lambda c_I} - f_{{\bf P}_0}(\lambda)\; |I|^{n - |{\bf P}_0|} \stackrel{(\ref{eq:step2})}= 0.
\end{eqnarray*}
That is, $({\tilde{\Theta'}}^n_I)_{\lambda c_I} (\delta_{\vec i}) = 0$ whenever ${\bf P}(\vec i) = {\bf P}_0$. If $\vec i$ is a multiindex with ${\bf P}(\vec i) < {\bf P}_0$, then ${\bf P}(\vec i) \in N(\{{\tilde{\Theta'}}^n_I\})$ by the minimality of ${\bf P}_0$, so that $\tilde{\Theta}^n_I(\delta_{\vec i}) = 0$. Moreover, $(\tau^{{\bf P}_0}_I)_{\lambda c_I}(\delta_{\vec i}) = 0$ by Lemma \ref{lem:can-tens-P}, whence
\[
({\tilde{\Theta'}}^n_I)_{\lambda c_I} (\delta_{\vec i}) = 0 \qquad \mbox{whenever} \qquad {\bf P}(\vec i) \leq {\bf P}_0,
\]
showing that ${\bf P}_0 \in N(\{{\tilde{\Theta'}}^n_I\})$. Therefore,
\[
N(\{\tilde{\Theta}^n_I\}) \subsetneq N(\{{\tilde{\Theta'}}^n_I\}).
\]

What we have shown is that given a congruent family of $n$-tensors $\{\tilde{\Theta}^n_I\}$ with $N(\{\tilde{\Theta}^n_I\}) \subsetneq \mbox{\bf Part}(n)$, we can enlarge $N(\{\tilde{\Theta}^n_I\})$ by subtracting a multiple of the canonical tensor of some partition. Repeating this finitely many times, we conclude that for some congruent family $\{ \tilde \Psi^n_I\}$ of the form (\ref{eq:congruent-form-rI})
\[
N(\{\tilde{\Theta}^n_I - \tilde \Psi^n_I\}) = \mbox{\bf Part}(n),
\]
and this implies by definition that $(\tilde{\Theta}^n_I - \tilde \Psi^n_I)_{\lambda c_I} = 0$ for all $I$ and all $\lambda > 0$.
\end{proof}

\begin{lemma} \label{lem:step3}
Let $\{ \tilde{\Theta}^n_I\; : \; I\; \mbox{finite}\}$ be a congruent family of $n$-tensors such that $(\tilde{\Theta}^n_I)_{\lambda c_I} = 0$ for all $I$ and $\lambda > 0$. Then $\tilde{\Theta}^n_I = 0$ for all $I$.
\end{lemma}

\begin{proof}
Consider $\mu \in \Mm_+(I)$ such that $\pi_I(\mu) = \mu/\|\mu\| \in \Pp_+(I)$ has rational coefficients, i.e.\ \[
\mu = \|\mu\| \sum_i \frac{k_i}n \delta_i
\]
for some $k_i, n \in \N$ and $\sum_{i \in I} k_i = n$. Let
\[
    I' \; := \; \biguplus_{i \in I} \left( \{ i \} \times \{1,\dots,k_i\} \right),
\]
so that $|I'|Ê= n$, and consider the congruent Markov kernel
\[
  K: \;\; i \; \longmapsto \; \frac{1}{k_i} \sum_{j = 1}^{k_i} \delta_{(i,j)}.
\]
Then
\[
K_*\mu = \|\mu\| \sum_i \frac{k_i}n \left(\frac{1}{k_i} \sum_{j = 1}^{k_i} \delta_{(i,j)}\right) = \|\mu\| \frac1n \sum_i \sum_{j = 1}^{k_i} \delta_{(i,j)} = \|\mu\| c_{I'}.
\]
Thus, Definition \ref{def:algebr-gen} implies
\[
(\tilde{\Theta}^n_I)_\mu (V_1, \ldots, V_n) = \underbrace{(\tilde{\Theta}^n_{I'})_{\|\mu\| c_{I'}}}_{=0} (K_*V_1, \ldots, K_*V_n) = 0,
\]
so that $(\tilde{\Theta}^n_I)_\mu = 0$ whenever $\pi_I(\mu)$ has rational coefficients. But these $\mu$ form a dense subset of $\Mm_+(I)$, whence $(\tilde{\Theta}^n_I)_\mu = 0$ for {\em all }$\mu \in \Mm_+(I)$, which completes the proof.
\end{proof}

We are now ready to prove the main result in this section.

\begin{proof}[Proof of Theorem \ref{thm:Chentsov-finite}] Let $\{ \tilde{\Theta}^n_I : I \mbox{ finite}\}$ be a congruent family of $n$-tensors. By Lemma \ref{lem:kill-center} there is a congruent family $\{ \tilde \Psi^n_I : I \mbox{ finite}\}$ of the form (\ref{eq:congruent-form-rI}) such that $(\tilde{\Theta}^n_I - \tilde \Psi^n_I)_{\lambda c_I} = 0$ for all finite $I$ and all $\lambda > 0$.

Since $\{ \tilde{\Theta}^n_I - \tilde \Psi^n_I : I \mbox{ finite}\}$ is again a congruent family, Lemma \ref{lem:step3} implies that $\tilde{\Theta}^n_I - \tilde \Psi^n_I  = 0$ and hence $\tilde{\Theta}^n_I = \tilde \Psi^n_I$ is of the form (\ref{eq:congruent-form-rI}), showing the statement of Theorem \ref{thm:Chentsov-finite} for $n$-tensors on $\Mm_+(I)$.

To show the second part, let us consider for a finite set $I$ the inclusion and projection
\[
\imath_I: \Pp_+(I) \hookrightarrow \Mm_+(I), \qquad \mbox{and} \qquad \pi_I: \begin{array}{rll} \Mm_+(I) & \longrightarrow & \Pp_+(I)\\[1mm] \mu & \longmapsto & \frac{\mu}{\|\mu\|} \end{array}
\]

Evidently, $\pi_I$ is a left inverse of $\imath_I$, i.e., $\pi_I \imath_I = Id_{\Pp_+(I)}$, and by (\ref{eq:Markov-preserve}) it follows that $K_\ast$ commutes both with $\pi_I$ and $\imath_I$.

Thus, if $\{\Theta^n_I : I \mbox{ finite}\}$ is a congruent family of $n$-tensors on $\Pp_+(I)$, then
\[
{\tilde \Theta}^n_I := \pi_I^\ast \Theta^n_I
\]
yields a congruent families of $n$-tensors on $\Mm_+(I)$ and by the first part of the theorem must be of the form (\ref{eq:congruent-form-rI}). But then,
\[
\Theta^n_I = \imath_I^\ast {\tilde \Theta}^n_I = \sum_{\bf P} c_{\bf P} (\tau^n_I)|_{\Pp_+(I)},
\]
where $c_{\bf P} = a_{\bf P}(1)$. Since $(\tau^n_I)|_{\Pp_+(I)} = 0$ if ${\bf P}$ contains a singleton set, it follows that $\Theta^n_I$ is of the form (\ref{eq:congruent-form-r-PI}).
\end{proof}

\section{Congruent families on arbitrary sample spaces} \label{sec:congruent-arb}

In this section, we wish to generalize the classification result for congruent families on finite sample spaces (Theorem \ref{thm:Chentsov-finite}) to the case of arbitrary sample spaces. As it turns out, we show that even in this case, congruent families of tensor fields are algebraically generated by the canonical tensor fileds. More precisely, we have the following result. 

\begin{theorem}[Classification of congruent families] \label{thm:chentsov-classif} 
For $0 < r \leq 1$, let $(\Theta^n_{\Om;r})$ be a family of covariant $n$-tensors on $\Mm^r(\Om)$ (on $\Pp^r(\Om)$, respectively) for each measurable space $\Om$. Then the following are equivalent:
\begin{enumerate}
\item $(\Theta^n_{\Om;r})$ is a congruent family of covariant $n$-tensors of regularity $r$.
\item For each congruent Markov morphism $K: I \to \Pp(\Om)$ for a finite set $I$, we have $K_r^*\Theta^n_{\Om;r} = \Theta^n_{I;r}$.
\item $\Theta^n_{\Om;r}$ is of the form (\ref{eq:congruent-form-r}) (of the form (\ref{eq:congruent-form-r-P}), respectively) for uniquely determined continuous functions $a_{\bf P}$ (constants $c_{\bf P}$, respectively).
\end{enumerate}
\end{theorem}

In the light of Definition \ref{def:alg-gen}, we may reformulate the equivalence of the first and the third statement as follows:

\begin{corollary}
The space of congruent families of covariant $n$-tensors on $\Mm^r(\Om)$ and $\Pp^r(\Om)$, respectively, is algebraically generated by the canonical $n$-tensors $\tau^n_{\Om;r}$ for $n \leq 1/r$.
\end{corollary}

\begin{proof}[Proof of Theorem \ref{thm:chentsov-classif}.]
We already showed in Proposition \ref{prop:alg-gen} that the tensors (\ref{eq:congruent-form-r}) and (\ref{eq:congruent-form-r-P}), respectively, are congruent families, hence the third statement implies the first. The first immediately implies the second by the definition of the congruency of tensors. Thus, it remains to show that the second statement implies the third.

We shall give the proof only for the families $(\Theta_{\Om;r}^n)$ of covariant $n$-tensors on $\Mm^r(\Om)$, as the proof for families on $\Pp^r(\Om)$ is analogous.

Observe that for finite sets $I$, the space $\Mm^r_+(I) \subset \Ss^r(I)$ is an open subset and hence a manifold, and the restrictions $\pi^\alpha: \Mm^r_+(I) \to \Mm_+^{r\alpha}(I)$ are diffeomorphisms not only for $\alpha \geq 1$ but for {\em all }$\alpha > 0$. Thus, given the congruent family $(\Theta^n_{\Om;r})$, we define for each finite set $I$ the tensor
\[
\Theta^n_I := (\pi^r)^*\Theta^n_{I;r} \quad \mbox{on $\Mm_+(I)$}.
\]

Then for each congruent Markov kernel $K: I \to \Pp(J)$ with $I$, $J$ finite we have
\begin{eqnarray*}
K^*\Theta^n_J & = & K^*(\pi^r)^*\Theta^n_{J;r} = (\pi^r K_*)^*\Theta^n_{J;r}\\
& \stackrel{(\ref{def:K*r})}= & (K_r \pi^r)^*\Theta^n_{J;r} = (\pi^r)^* K_r^* \Theta^n_{J;r} = (\pi^r)^*\Theta^n_{I;r}\\
& = & \Theta^n_I.
\end{eqnarray*}

Thus, the family $(\Theta^n_I)$ on $\Mm_+(I)$ is a congruent family of covariant $n$-tensors on finite sets, whence by Theorem \ref{thm:Chentsov-finite}
\[
(\Theta^n_I)_\mu = \sum_{{\bf P} \in {\bf Part}(n)} a_{\bf P}(\|\mu\|) (\tau^{\bf P}_I)_\mu
\]
for uniquely determined functions $a_{\bf P}$, whence on $\Mm_+^r(I)$,
\begin{eqnarray*}
\Theta^n_{I;r} & = & (\pi^{1/r})^*\Theta^n_I\\
& = & \sum_{{\bf P} \in {\bf Part}(n)} a_{\bf P}(\|\mu_r^{1/r}\|) (\pi^{1/r})^*\tau^{\bf P}_I\\
& = & \sum_{{\bf P} \in {\bf Part}(n)} a_{\bf P}(\|\mu_r^{1/r}\|) \tau^{\bf P}_{I;r}.
\end{eqnarray*}

By our assumption, $\Theta^n_{I;r}$ must be a covariant $n$-tensor on $\Mm(I)$, whence it must extend continuously to the boundary of $\Mm_+(I)$.

But by (\ref{eq:can-tens-comp}) it follows that $\tau^{n_i}_{I;r}$ has a singularity at the boundary of $\Mm(I)$, unless $n_i \leq 1/r$. From this it follows that $\Theta^n_{I;r}$ extends to all of $\Mm(I)$ if and only if $a_{\bf P} \equiv 0$ for all partitions ${\bf P} = \{P_1, \ldots, P_i\}$ where $|P_i| > 1/r$ for some $i$.

Thus, $\Theta^n_{I;r}$ must be of the form (\ref{eq:congruent-form-r})
for all finite sets $I$. Let
\[
\Psi^n_{\Om;r} := \Theta^n_{\Om;r}- \sum_{\bf P} a_{\bf P}(\|\mu_r^{1/r}\|) \tau^n_{\Om;r}
\]
for the previously determined functions $a_{\bf P}$, so that $(\Psi^n_{\Om;r})$ is a congruent family of covariant $n$-tensors, and $\Psi^n_{I;r} = 0$ for every finite $I$.

We assert that this implies that $\Psi^n_{\Om;r} = 0$ for {\em all }$\Om$, which shows that $\Theta^n_{\Om;r}$ is of the form (\ref{eq:congruent-form-r}) for all $\Om$, which will complete the proof.

To see this, let $\mu_r \in \Mm^r(\Om)$ and $\mu := \mu_r^{1/r} \in \Mm(\Om)$. Moreover, let $V_j = \phi_j \mu_r \in \Ss^r(\Om, \mu_r)$, $j = 1, \ldots, n$, such that the $\phi_j$ are step functions. That is, there is a finite partition $\Om = \dot \bigcup_{i \in I} \Om_i$ such that
\[
\phi_j = \sum_{i \in I} \phi_j^i \chi_{\Om_i}
\]
for $\phi_j^i \in \R$ and $m_i := \mu(\Om_i) > 0$.

Let $\kappa: \Om \to I$ be the statistic $\kappa(\Om_i) = \{i\}$, and $K: I \to \Pp(\Om)$, $K(i) := 1/m_i \chi_{\Om_i} \mu$. Then clearly, $K$ is $\kappa$-congruent, and $\mu = K_*\mu'$ with $\mu' := \sum_{i \in I} m_i \delta_i \in \Mm_+(I)$. Thus, by (\ref{eq:K-mu-r})
\[
d_{\mu'}K_r\left(\sum_{i \in I} \phi_j^i m_i^r \delta_i^r\right) = \sum_{i \in I} \phi_j^i \chi_{\Om_i} \mu_r = \phi_j \mu_r = V_j,
\]
whence if we let $V_j' := \sum_{i \in I} \phi_j^i m_i^r \delta_i^r \in \Ss^r(I)$, then
\begin{eqnarray*}
\Psi_{\Om;r}^n(V_1, \ldots, V_n) & = & \Psi_{\Om;r}^n(dK_r(V'_1), \ldots, dK_r(V'_n))\\
& = & K_r^* \Psi_{\Om;r}^n (V'_1, \ldots, V'_n)\\
& = & \Psi_{I;r}^n (V'_1, \ldots, V'_n) = 0,
\end{eqnarray*}
since by the congruence of the family $(\Psi_{\Om;r}^n)$ we must have $K_r^* \Psi_{\Om;r}^n = \Psi_{I;r}^n$, and $\Psi_{I;r}^n = 0$ by assumption as $I$ is finite.

That is, $\Psi_{\Om;r}^n(V_1, \ldots, V_n) = 0$ whenever $V_j = \phi_j \mu_r \in \Ss^r(\Om, \mu_r)$ with step functions $\phi_j$. But the elements $V_j$ of this form are dense in $\Ss^r(\Om, \mu_r)$, hence the continuity of $\Psi_{\Om;r}^n$ implies that $\Psi_{\Om;r}^n = 0$ for all $\Om$ as claimed.
\end{proof}

As two special cases of this result, we obtain the following.

\begin{corollary}[Generalization of Chentsov's theorem]   \label{cor:chentsov-general}
\begin{enumerate}
\item
Let $(\Theta_\Om^2)$ be a congruent family of $2$-tensors on $\Pp^{1/2}(\Om)$. Then up to a constant, this family is the Fisher metric. That is, there is a constant $c \in \R$ such that for all $\Om$,
\[
\Theta_\Om^2 = c\; \g_F.
\]
In particular, if $(M, \Om, \pb)$ is a $2$-integrable statistical model, then
\[
\pb^*\Theta_\Om^2 = c\; \g_M
\]
is -- up to a constant -- the Fisher metric of the model.
\item
Let $(\Theta_\Om^3)$ be a congruent family of $3$-tensors on $\Pp^{1/3}(\Om)$. Then up to a constant, this family is the Amari--Chentsov tensor. That is, there is a constant $c \in \R$ such that for all $\Om$,
\[
\Theta_\Om^3 = c\; \T.
\]
In particular, if $(M, \Om, \pb)$ is a $3$-integrable statistical model, then
\[
\pb^*\Theta_\Om^3 = c\; \T_M
\]
is -- up to a constant -- the Amari--Chentsov tensor of the model.
\end{enumerate}
\end{corollary}

\begin{corollary}[Generalization of Campbell's theorem] \label{cor:campbell-general}
Let $(\Theta_\Om^2)$ be a congruent family of $2$-tensors on $\Mm^{1/2}(\Om)$. Then there are continuous functions $a, b: (0, \infty) \to \R$ such that
\[
(\Theta_\Om^2)_{\mu^{1/2}} (V_1, V_2) = a(\|\mu\|) \g_F(V_1, V_2) + b(\|\mu\|) \tau^1_{\Om;1/2}(V_1) \tau^1_{\Om;1/2}(V_2).
\]

In particular, if $(M, \Om, \pb)$ is a $2$-integrable parametrized measure model, then
\begin{eqnarray*}
\pb^*(\Theta_\Om^2)_\xi(V_1, V_2) & = & a(\|\pb(\xi)\|) \int_\Om \partial_{V_1} \log \pb(\xi)\; \partial_{V_2} \log \pb(\xi)\; d\pb(\xi)\\
& & \qquad \qquad + b(\|\pb(\xi)\|) \left(\partial_{V_1} \|\pb(\xi)\| \right) \left( \partial_{V_2} \|\pb(\xi)\| \right).
\end{eqnarray*}
\end{corollary}

While the above results show that for small $n$ there is a unique family of congruent $n$-tensors, this is no longer true for larger $n$. For instance, for $n = 4$ Theorem \ref{thm:chentsov-classif} implies that any restricted congruent family of invariant $4$-tensors on $\Pp^r(\Om)$, $0 < r \leq 1/4$, is of the form
\begin{eqnarray*}
\Theta_\Om^4(V_1, \ldots, V_4) & = & c_0 \tau^4_{\Om;r}(V_1, \ldots, V_4)\\
& & + c_1 \tau^2_{\Om;r}(V_1, V_2) \tau^2_{\Om;r}(V_3, V_4)\\
& & + c_2 \tau^2_{\Om;r}(V_1, V_3) \tau^2_{\Om;r}(V_2, V_3)\\
& & + c_3 \tau^2_{\Om;r}(V_1, V_4) \tau^2_{\Om;r}(V_2, V_4),
\end{eqnarray*}
so that the space of congruent families on $\Pp^r(\Om)$ is already $4$-dimensional in this case. Evidently, this dimension rapidly increases with $n$.


\end{document}